\documentclass[11pt]{article}
\usepackage{amsfonts,amssymb,amsmath,amsthm}
\usepackage{color}

\numberwithin{equation}{section}

\usepackage{fullpage,slashbox,graphicx}

\definecolor{darkgreen}{rgb}{0.01,0.75,0.24}

\newcommand{\EE}{\mathbb{E}}
\newcommand{\VV}{\mathbb{V}}
\newcommand{\R}{\mathbb{R}}
\newcommand{\N}{\mathbb{N}}

 \newtheorem{theorem}{Theorem}[section]
  \newtheorem{proposition}[theorem]{Proposition}
  \newtheorem{corollary}[theorem]{Corollary}
  \newtheorem{lemma}[theorem]{Lemma}
\newtheorem{remark}[theorem]{Remark}
\newtheorem{definition}[theorem]{Definition}
\newtheorem{assumption}[theorem]{Assumption}

\DeclareMathOperator*{\argmin}{arg\,min}

\newcommand{\dhh}{d_{\mbox {\tiny{\rm Hell}}}}
\newcommand{\dtv}{d_{\mbox {\tiny{\rm TV}}}}

\begin{document}
\title{Posterior Consistency for Gaussian Process Approximations of Bayesian Posterior Distributions}
\author{Andrew M. Stuart$^1$, Aretha L. Teckentrup$^1$}
\date{}
\maketitle

\noindent
$^1$ Mathematics Institute, Zeeman Building, University of Warwick, Coventry, CV4 7AL, England. \texttt{a.m.stuart@warwick.ac.uk, a.teckentrup@warwick.ac.uk}

\begin{abstract}
We {study} the use of Gaussian process emulators to approximate the parameter-to-observation map or the negative log-likelihood in Bayesian inverse problems. We prove error bounds on the Hellinger distance between the true posterior distribution and various approximations based on the Gaussian process emulator. Our analysis includes approximations based on the mean of the predictive process, as well as approximations based on the full Gaussian process emulator. Our results show that the Hellinger distance between the true posterior and its approximations can be bounded by moments of the error in the emulator. Numerical results confirm our theoretical findings.
\end{abstract}

 \vspace{2ex}


\section{Introduction}
{Given a} {mathematical model} of a physical process, we are interested in the inverse problem of determining the inputs to the model given some noisy observations related to the model outputs. Adopting a Bayesian {approach 
\cite{kaipio2005statistical,stuart10}}, we incorporate our prior knowledge of the inputs into a probability distribution, referred to as the {\em prior distribution}, and obtain a more accurate representation of the model inputs in the {\em posterior distribution}, which results from conditioning the prior distribution on the observations. 
Since the posterior distribution is generally intractable, sampling methods such as Markov chain Monte Carlo (MCMC) \cite{hastings70,mrrtt53,robert_casella,cmps14,gc11,crsw13} are typically used {to explore it}. A major challenge in the application of MCMC methods to problems of practical interest is the large computational cost associated with numerically solving the mathematical model for a given set of the input parameters. Since the generation of each sample by the MCMC method requires a solve of the governing equations, and often millions of samples are required, this process can quickly become very costly. 

{This drawback of fully Bayesian inference for complex models was
recognised several decades ago in the statistics literature, and
resulted in key papers which had a profound influence on methodology 
\cite{sacks1989design,kennedy2001bayesian,o2006bayesian}. 
These papers advocated
the use of a Gaussian process surrogate model to approximate the solution of the 
governing equations, and in particular the likelihood,
at a much lower computational cost. This approximation then results in an approximate posterior distribution, which can be sampled more cheaply using MCMC. 
However, despite the widespread adoption of the methodology, there has been
little analysis of the effect of the approximation on posterior
inference. In this work, we study this issue, focussing on the use of Gaussian process emulators \cite{rasmussen_williams,stein,sacks1989design,kennedy2001bayesian,o2006bayesian,brsrwm08,hkccr04} as surrogate models. 
Other choices of surrogate models such as those described in \cite{bwg08,akksstv06}, generalised Polynomial Chaos \cite{xk03,mnr07}, sparse grid collocation \cite{bnt10,mx09} and
adaptive subspace methods \cite{constantine2014active,constantine2015active} 
might also be studied similarly, but are not considered here. Indeed 
we note that the paper \cite{mx09} studied the effect, on the
posterior distribution, of stochastic collocation approximation within
the forward model and was one of the first papers to address such questions.
That paper used the Kullback-Leibler divergence, or relative entropy,
to measure the effect on the posterior, and considered finite dimensional
input parameter spaces. 
}

{ The main focus of this work is to analyse the error introduced in the posterior distribution by using a Gaussian process emulator as a surrogate model. The error is measured in the Hellinger distance, which {is shown in \cite{stuart10,ds15} to be a suitable metric for evaluation of perturbations to the posterior measure in Bayesian inverse problems, including problems with infinite dimensional input parameter spaces. We consider emulating either the parameter-to-observation map or the negative log-likelihood.} The convergence results presented in this paper are of two types. In section \ref{sec:gp}, we present convergence results for simple Gaussian process emulators applied to a general function $f$ satisfying suitable regularity assumptions. In section \ref{sec:gp_app}, we prove bounds on the error in the posterior distribution in terms of the error in the Gaussian process emulator. The novel contributions of this work are mainly in section \ref{sec:gp_app}. The results in the two sections can be combined to give a final error estimate for the simple Gaussian process emulators presented in section \ref{sec:gp}. However, the error bounds derived in section \ref{sec:gp_app} are much more general in the sense that they apply to any Gaussian process emulator satisfying the required assumptions. A short discussion on extensions of this work related to Gaussian process emulators used in practice is included in the conclusions in section \ref{sec:conc}.}

{ We study three different approximations to the posterior distribution. Firstly, we consider using the mean of the Gaussian process emulator as a surrogate model, resulting in a deterministic approximation to the posterior distribution. Our second approximation is obtained by using the full Gaussian process as a surrogate model, leading to a random approximation in which case we study the second moment of the Hellinger distance between the true and the approximate posterior distribution. The uncertainty in the posterior distribution introduced in this way can be thought of representing the uncertainty in the emulator due to the finite number of function evaluations used to construct it. This uncertainty can in applications be large (or comparable) to the uncertainty present in the observations, and a user may want to take this into account to "inflate" the variance of the posterior distribution. Finally, we construct an alternative deterministic approximation by using the full Gaussian process as surrogate model, and taking the expected value (with respect to the distribution of the surrogate) of the likelihood. It can be shown that this approximation of the likelihood is optimal in the sense that it minimises the $L^2$-error \cite{sn16}. In contrast to the approximation based on only the mean of the emulator, this approximation also takes into account the uncertainty of the emulator, although only in an averaged sense. 
}

For the three approximations discussed above, we show that the Hellinger distance between the true and approximate posterior distribution can be bounded by the error between the true parameter-to-observation map (or log-likelihood) and its Gaussian process approximation, measured in a norm that depends on the approximation considered. {Our analysis is restricted to finite dimensional input spaces. This
reflects the state-of-the-art with respect to Gaussian process
emulation itself; the analysis of the effect on the posterior is 
less sensitive to dimension.} {For simplicity, we also restrict our attention to bounded parameters, i.e. parameters in a compact subset of $\R^K$ for some $K \in \N$, and to problems where the parameter-to-observation map is uniformly bounded.}

{ The convergence results on Gaussian process regression presented in section \ref{sec:gp} are mainly known results from the theory of scattered data interpolation \cite{wendland,sss13,nww06}. The error bounds are given in terms of the fill distance of the design points used to construct the Gaussian process emulator, and depend in several ways on the number $K$ of input parameters we want to infer. Firstly, when looking at the error in terms of the number of design points used, rather than the fill distance of these points, the rate of convergence typically deteriorates with the number of parameters $K$. Secondly, the proof of these error estimates requires assumptions on the smoothness of the function being emulated, where the precise smoothness requirements depend on the Gaussian process emulator employed. For emulators based on Mat\`ern kernels \cite{matern}, we require these maps to be in a Sobolev space $H^s$, where $s > K/2$. We would like to point out here that it is not necessary for the function being emulated to be in the {\em reproducing kernel Hilbert space} (or {\em native space}) of the Mat\`ern kernel used in order to prove convergence (cf Proposition \ref{prop:mean_conv_int}), but that is suffices to be in a larger Sobolev space in which point evaluations are bounded linear functionals. 
}

The remainder of this paper is organised as follows. In section \ref{sec:inv}, we set up the Bayesian inverse problem of interest. We then recall some results on Gaussian process regression in section \ref{sec:gp}. The heart of the paper is section \ref{sec:gp_app}, where we introduce the different approximations to the posterior and perform an error analysis. Our theoretical results are confirmed on a simple model problem in section \ref{sec:num}, and some conclusions are finally given in section \ref{sec:conc}.

\section{Bayesian Inverse Problems}\label{sec:inv}
%A detailed mathematical treatment of Bayesian inverse problems in the function space setting can be found in \cite{ds14}.
{Let $X \subseteq \R^K$, for some $K \in \N$, represent the range of a finite number $K$ of parameters $u$, and define the measurable mappings $G: \R^K \rightarrow V$ and $\mathcal O : V \rightarrow \mathbb R^J$, for some separable Banach space $V$ and $J \in \mathbb N$. Denote by $\mathcal G: \R^K \rightarrow \mathbb R^J$ the composition of $\mathcal O$ and $G$. We refer to $G$ as the {\em forward map}, to $\mathcal O$ as the {\em observation operator} and to $\mathcal G$ as the {\em parameter-to-observation map}. We denote by $\|\cdot\|$ the Euclidean norm on $\mathbb R^n$, for $n \in \N$.  } 
%a space of real-valued functions defined on a bounded spatial domain $D \subset \mathbb R^d$, for some dimension $d = 1, 2$ or $3$. For ease of presentation, we shall restrict our attention to the case $V = C(\overline D)$, the space of continuous functions on $\overline D$, but other choices are possible, as is the extension to vector-valued functions.
The inverse problem of interest is to determine the parameters $u \in X$ from the noisy data $y \in \R^J$ given by
\begin{equation*}
y = \mathcal G(u) + \eta,
\end{equation*}
where the noise $\eta$ is a realisation of the $\mathbb R^J$-valued Gaussian random variable $\mathcal N(0,\sigma_\eta^2 I)$, for some known variance $\sigma_\eta^2$. 

We adopt a Bayesian perspective in which, in the absence of data, $u$ is distributed according to a {prior measure $\mu_0$ on $\R^K$, with $\mu_0(X) = 1$. We make the following further assumption on $\mu_0$.} \vspace{1ex}

{ \noindent {\bf Assumption A.} We have $\|f\|_{L^2_{\mu_0}(X)} \leq C_{\mathrm A}\|f\|_{L^2(X)}$, for all $f \in L^2(X)$. \vspace{1ex} }

We are interested in the posterior distribution $\mu^y$ on the conditioned random variable $u | y$, which can be characterised as follows.

\begin{proposition} (\cite{stuart10}) Suppose $\mathcal G : { \R^K} \rightarrow \R^J$ is continuous. Then the posterior distribution $\mu^y$ on the conditioned random variable $u | y$ is absolutely continuous with respect to $\mu_0$ and given by Bayes' Theorem: 
\begin{equation*}\label{eq:rad_nik}
\frac{d\mu^y}{d\mu_0}(u) = \frac{1}{Z} \exp\big(-\Phi(u)\big),
\end{equation*}
where
\begin{equation}\label{eq:def_like}
\Phi(u) = \frac{1}{2 \sigma_\eta^2} \left\| y -  \mathcal G (u) \right\|^2  \quad \text{and } \qquad Z = \EE_{\mu_0}\Big(\exp\big(-\Phi(u)\big)\Big).
\end{equation}
\end{proposition}

We make the following assumption on the regularity of the parameter-to-observation map $\mathcal G$.

\begin{assumption}\label{ass:reg} {We assume that}
$\mathcal G : X \rightarrow \R^J$ satisfies $\mathcal G \in H^s(X; \R^J)$, for some $s > K/2$, {and that} $\sup_{u \in X} \|\mathcal G(u)\| =: C_\mathcal G < \infty$.
\end{assumption}

Under Assumption \ref{ass:reg}, it follows that the negative log-likelihood $\Phi : X \rightarrow \R$ satisfies $\Phi \in H^s(X)$, and  $\sup_{u \in X} |\Phi(u)| =: C_\Phi < \infty$. Since $s > K/2$, the Sobolev Embedding Theorem furthermore implies that $\mathcal G$ and $\Phi$ are continuous.
Examples of model problems satisfying Assumption \ref{ass:reg} include linear elliptic and parabolic partial differential equations \cite{cohen2011analytic,ss14} and non-linear ordinary differential equations \cite{walter,hs13}. A specific example is given in section \ref{sec:num}. 

{ Note that in Assumption \ref{ass:reg}, the smoothness requirement on $\mathcal G$ becomes stronger as $K$ increases. The reason for this is that in order to apply the results in section \ref{sec:gp}, we require $\mathcal G$ to be in a Sobolev space in which point evaluations are bounded linear functionals. The second part of Assumption \ref{ass:reg} is mainly included to define the constant $C_{\mathcal G}$, since the fact that $\sup_{u \in X} \|\mathcal G(u)\|$ is finite follows from the continuity of $\mathcal G$ and the compactness of $X$.}

\section{Gaussian {Process Regression}}\label{sec:gp}
We are interested in using Gaussian process regression to build a surrogate model for the forward map, leading to an approximate Bayesian posterior distribution that is computationally cheaper to evaluate. Generally speaking, Gaussian process regression (or Gaussian process emulation, or kriging) is a way of building an approximation to a function $f$, based on a finite number of evaluations of $f$ at a chosen set of {design points}. We will here consider emulation of either the parameter-to-observation map $\mathcal G: X \rightarrow \R^J$ or the negative log-likelihood $\Phi:X \rightarrow \R$. Since the efficient emulation of vector-valued functions is still an open question \cite{bzkl13}, we will focus on the emulation of scalar valued functions. An emulator of $\mathcal G$ in the case $J > 1$ is constructed by emulating each entry independently.

Let now $f : X \rightarrow \R$ be an arbitrary function. Gaussian process emulation is in fact a Bayesian procedure, and the starting point is to put a Gaussian process prior on the function $f$. In other words, we model $f$ as
\begin{equation}\label{eq:gp}
{f_0} \sim \text{GP}(m(u), k(u,u')),
\end{equation}
with known mean $m : X \rightarrow \R$ and two point covariance function $k : X \times X \rightarrow \R$, {assumed to be positive-definite.} Here, we use the Gaussian process notation as in, for example, \cite{rasmussen_williams}. In the notation of \cite{stuart10}, we have ${f_0} \sim \mathcal N(m,C)$, where $m=m(\cdot)$ and $C$ is the integral operator with covariance function $k$ as kernel.

Typical choices of the mean function $m$ include the zero function and polynomials \cite{rasmussen_williams}. 
A family of covariance functions $k$ frequently used in applications are the Mat\`ern covariance functions \cite{matern}, given by
\begin{equation}\label{eq:mat_cov}
k_{\nu,\lambda,\sigma_k^2}(u,u') = \sigma_k^2 \, \frac{1}{\Gamma(\nu) 2^{\nu-1}} \left(\sqrt{2\nu} \frac{\|u-u'\|}{\lambda}\right)^\nu B_\nu\left(\sqrt{2\nu} \frac{\|u-u'\|}{\lambda}\right) ,
\end{equation}
where $\Gamma$ denotes the Gamma function, $B_\nu$ denotes the modified Bessel function of the second kind and $\nu, \lambda$ and $\sigma_k^2$ are positive parameters. The parameter $\lambda$ is referred to as the {\em correlation length}, and governs the length scale at which ${f_0}(u)$ and ${f_0}(u')$ are correlated. The parameter $\sigma_k^2$ is referred to as the {\em variance}, and governs the magnitude of ${f_0}(u)$. Finally, the parameter $\nu$ is referred to as the {\em smoothness parameter}, and governs the regularity of ${f_0}$ as a function of $u$.
As the limit when $\nu \rightarrow \infty$, we obtain the Gaussian covariance
\begin{equation}\label{eq:gauss_cov}
k_{\infty,\lambda,\sigma_k^2}(u,u') = \sigma_k^2 \exp \left(-\frac{\|u-u'\|^2}{2 \lambda^2}\right).
\end{equation} 

Now suppose we are given data in the form of a set of distinct {\em design points} $U := \{u^n\}_{n=1}^N \subseteq X$, together with corresponding function values 
\begin{equation}\label{eq:data_exact}
f(U) := [f(u^1), \dots, f(u^N)] \in \R^N.
\end{equation}
{Since $f_0$ is a Gaussian process, the vector $[f_0(u^1), \dots, f_0(u^N), f_0(\tilde u^1), \dots, f_0(\tilde u^M)] \in R^{N+M}$, for any set of test points $\{ \tilde u^m\}_{m=1}^M \subseteq X \setminus U$, follows a multivariate Gaussian distribution. The conditional distribution of $f_0(\tilde u^1), \dots, f_0(\tilde u^M)$, given the values $f_0(u^1) = f(u^1), \dots, f_0(u^N) = f(u^N)$, is then again Gaussian, with mean and covariance given by the standard formulas for the conditioning of Gaussian random variables \cite{rasmussen_williams}.} 

Conditioning the Gaussian process \eqref{eq:gp} on the known values $f(U)$, we hence obtain another Gaussian process $f_N$, known as the {\em predictive process}. We have
\begin{equation}\label{eq:gp_pred}
f_N \sim \text{GP}(m^f_N(u), k_N(u,u')),
\end{equation}
where the predictive mean $m^f_N : X \rightarrow \R$ and predictive covariance $k_N : X \times X \rightarrow \R$ are known explicitly, and depend on the modelling choices made in \eqref{eq:gp}. In the following discussion, we will focus on the popular choice $m \equiv 0$; the case of a non-zero mean is discussed in Remark \ref{rem:mean}. 
When $m \equiv 0$, we have
\begin{align}\label{eq:pred_eq}
m_N^f(u) = k(u,U)^T K(U,U)^{-1} f(U), \qquad
k_N(u,u') = k(u,u') - k(u,U)^T K(U,U)^{-1} k(u',U),
\end{align}
where $k(u,U) = [k(u,u^1), \dots, k(u,u^N)] \in \R^{N}$ and $K(U,U) \in \R^{N \times N}$ is the matrix with $ij^\mathrm{th}$ entry equal to $k(u^i,u^j)$ \cite{rasmussen_williams}. 

There are several points to note about the predictive mean $m_N^f$ in \eqref{eq:pred_eq}. Firstly, $m_N^f$ is a linear combination of the function evaluations $f(U)$, and hence a linear predictor. It is in fact the {\em best linear predictor} \cite{stein}, in the sense that it is the linear predictor with the smallest mean square error. Secondly, $m_N^f$ interpolates the function $f$ at the design points $U$, since the vector $k(u^n,U)$ is the $n^\mathrm{th}$ row of the matrix $K(U,U)$. In other words, we have $m_N^f(u^n) = f(u^n)$, for all $n=1,\dots,N$. 
Finally, we remark that $m_N^f$ is a linear combination of kernel evaluations,
\begin{equation*}%\label{eq:pred_mean}
m_N^f(u) = \sum_{n=1}^N \alpha_n k(u,u^n),
\end{equation*}
where the vector of coefficients is given by $\alpha = K(U,U)^{-1} f(U)$. 
Concerning the predictive covariance $k_N$, we note that $k_N(u,u) < k(u,u)$ for all $u \in X$, since $K(U,U)^{-1}$ is positive definite. Furthermore, we also note that $k_N(u^n,u^n) = 0$, for $n=1, \dots, N$, since $k(u^n,U)^T \; K(U,U)^{-1} \; k(u^n,U) = k(u^n,u^n)$.

For stationary covariance functions $k(u,u') = k(\|u-u'\|)$, the predictive mean is a radial basis functions interpolant of $f$, and we can make use of results from the radial basis function literature to investigate the behaviour of $m_N^f$ and $k_N$ as $N \rightarrow \infty$. Before we do this, {in
subsection \ref{ssec:gp_sk}, we recall some results on native spaces 
(also know as reproducing kernel Hilbert spaces) in subsection \ref{ssec:gp_native}.}

\subsection{Native spaces of Mat\`ern kernels}\label{ssec:gp_native}
We recall the notion of the reproducing kernel Hilbert space corresponding to the kernel $k$, usually referred to as the native space of $k$ in the radial basis function literature.

\begin{definition}\label{def:rkhs} A Hilbert space $H_k$ of functions $f: X \rightarrow \R$, with inner product $\langle \cdot, \cdot \rangle_{H_k}$, is called the {\em reproducing kernel Hilbert space (RKHS)} corresponding to a symmetric, positive definite kernel $k : X \times X \rightarrow \R$ if
\begin{itemize}
\item[i)] for all $u \in X$, $k(u, u')$, as a function of its second argument, belongs to $H_k$,
\item[ii)] for all $u \in X$ and $f \in H_k$, $\langle f, k(u, \cdot) \rangle_{H_k} = f(u)$.
\end{itemize}
\end{definition}

By the Moore-Aronszajn Theorem \cite{aronszajn50}, a unique RKHS exists for each symmetric, positive definite kernel $k$. Furthermore, this space can be constructed using Mercer's Theorem \cite{mercer09}, and it is equal to the Cameron-Martin space \cite{bogachev} of the covariance operator $C$ with kernel $k$.
For covariance kernels of Mat\`ern type, the native space is isomorphic to a Sobolev space \cite{wendland,sss13}.

\begin{proposition}\label{prop:native_matern} Let $k_{\nu,\lambda,\sigma_k^2}$ be a Mat\`ern covariance kernel as defined in \eqref{eq:mat_cov}. Then the native space $H_{k_{\nu,\lambda,\sigma_k^2}}$ is equal to the Sobolev space $H^{\nu+K/2}(X)$ as a vector space, and the native space norm and the Sobolev norm are equivalent. 
\end{proposition}

Native spaces for more general kernels, including non-stationary kernels, are analysed in \cite{wendland}. For stationary kernels, the native space can generally be characterised by the rate of decay of the Fourier transform of the kernel. The native space of the Gaussian kernel \eqref{eq:gauss_cov}, for example, consists of functions whose Fourier transform decays exponentially, and is hence strictly contained in the space of analytic functions.
Proposition \ref{prop:native_matern} shows that as a vector space, the native space of the Mat\`ern kernel $k_{\nu,\lambda,\sigma_k^2}$ is fully determined by the smoothness parameter $\nu$. The parameters $\lambda$ and $\sigma_k^2$ do, however, influence the constants in the norm equivalence of the native space norm and the standard Sobolev norm.

\subsection{Radial basis function interpolation}\label{ssec:gp_sk}
For stationary covariance functions $k(u,u') = k(\|u-u'\|)$, the predictive mean is a radial basis functions interpolant of $f$. In fact, it is the minimum norm interpolant \cite{rasmussen_williams},
\begin{equation}\label{eq:mean_min}
m_N^f = \argmin_{g \in H_k \; : \; g(U) = f(U)} \|g\|_{H_k}.
\end{equation}
Given the set of design points $U = \{u^n\}_{n=1}^N \subseteq X$, we define the {fill distance} $h_U$, {separation radius} $q_U$ and {mesh ratio} $\rho_U$ by
\begin{equation*}
h_{U} := \sup_{u \in X} \inf_{u^n \in U} \|u-u^n\|, \qquad q_U := \frac{1}{2} \min_{i \neq j} \|u^j - u^i\|, \qquad \rho_U := \frac{h_U}{q_U} \geq 1.
\end{equation*}
The fill distance is the maximum distance any point in $X$ can be from $U$, and the separation radius is half the smallest distance between any two distinct points in $U$. The mesh ratio provides a measure of how uniformly the design points $U$ are distributed in X. 
We have the following theorem on the convergence of $m_N^f$ to $f$ \cite{wendland,nww05,nww06}.
 
\begin{proposition}\label{prop:mean_conv} Suppose $X \subseteq \mathbb R^K$ is a bounded, Lipschitz domain that satisfies an interior cone condition, and the symmetric positive definite kernel $k$ is such that $H_k$ is isomorphic to the Sobolev space $H^\tau(X)$, with $\tau = n + r$, $n \in \mathbb N$, $n > K/2$ and $0 \leq r < 1$. Suppose $m_N^f$ is given by \eqref{eq:pred_eq}. If $f \in H^\tau(X)$, then there exists a constant $C$, independent of $f$, $U$ and $N$, such that
\[
\| f - m_N^f\|_{H^\beta(X)} \leq C h_U^{\tau - \beta} \|f\|_{H^\tau(X)}, \qquad \text{for any } \beta \leq \tau,
\]
for all sets $U$ with $h_U$ sufficiently small.
\end{proposition}

Proposition \ref{prop:mean_conv} assumes that the function $f$ is in the RKHS of the kernel $k$. Convergence estimates for a wider class of functions can be obtained using interpolation in Sobolev spaces \cite{nww06}. 

\begin{proposition}\label{prop:mean_conv_int} Suppose $X \subseteq \mathbb R^K$ is a bounded, Lipschitz domain that satisfies an interior cone condition, and the symmetric positive definite kernel $k$ is such that $H_k$ is isomorphic to the Sobolev space $H^\tau(X)$. Suppose $m_N^f$ is given by \eqref{eq:pred_eq}. If $f \in H^{\tilde \tau}(X)$, for some $\tilde \tau \leq \tau$, $\tilde \tau = n + r$, $n \in \mathbb N$, $n > K/2$ and $0 \leq r < 1$, then there exists a constant $C$, independent of $f$, $U$ and $N$, such that
\[
\| f - m_N^f\|_{H^\beta(X)} \leq C h_U^{\tilde \tau - \beta} \rho_U^{\tau - \beta}\|f\|_{H^{\tilde \tau}(X)}, \qquad \text{for any } \beta \leq \tilde \tau,
\]
for all sets $U$ with $h_U$ and $\rho_U$ sufficiently small.
\end{proposition}

{ We would like to point out here that in practice, it is much more informative to obtain convergence rates in terms of the number of design points $N$ rather than their associated fill distance $h_U$. This is of course possible in general, but the precise relation between $N$ and $h_U$ will depend on the specific choice of design points $U$. For uniform tensor grids $U$, the fill distance $h_U$ is of the order $N^{-1/K}$ (cf section \ref{sec:num}). This suggests a strong dependence on the input dimension $K$ of the convergence rate in terms of the number of design points $N$.}

Convergence of the predictive variance $k_N(u,u)$ follows under the assumptions of Proposition \ref{prop:mean_conv} or Proposition \ref{prop:mean_conv_int} using the relation in Proposition \ref{prop:predvar_sup} below. This was already noted, {without proof,} in \cite{sss13}; 
we give a proof here for completeness.

\begin{proposition}\label{prop:predvar_sup} Suppose $m_N^f$ and $k_N$ are given by \eqref{eq:pred_eq}. Then
\[
k_N(u,u)^{\frac{1}{2}} = \sup_{\|g\|_{H_k}=1} | g(u) - m^g_N(u)|.
\]
\end{proposition}
\begin{proof} For any $u \in X$, we have 
\begin{align*}
\sup_{\|g\|_{H_k}=1} | g(u) - m^g_N(u)| &= \sup_{\|g\|_{H_k}=1} \Big| g(u) - \sum_{j=1}^N (k(u,U)^T K(U,U)^{-1} )_j g(u^j) \Big| \\
&= \sup_{\|g\|_{H_k}=1} \Big|\left \langle g, k(\cdot,u) \right\rangle_{H_k} - \sum_{j=1}^N (k(u,U)^T K(U,U)^{-1} )_j  \langle g, k(\cdot,u^j) \rangle_{H_k} \Big| \\
&=  \sup_{\|g\|_{H_k}=1} \Big|\langle g, k(\cdot,u) - \sum_{j=1}^N (k(u,U)^T K(U,U)^{-1} )_j  k(\cdot,u^j) \rangle_{H_k} \Big| \\
&= \| k(\cdot,u) - k(\cdot,U)^T K(U,U)^{-1} k(u,U)\|_{H_k}.
\end{align*}
The final equality follows from the Cauchy-Schwarz inequality, which becomes an equality when the two functions considered are linearly dependent.
By Definition \ref{def:rkhs}, we then have
\begin{align*}
&\| k(\cdot,u) - k(\cdot,U)^T K(U,U)^{-1} k(u,U)\|_{H_k}^2 \\
& \qquad = \langle k(\cdot,u) - k(\cdot,U)^T K(U,U)^{-1} k(u,U), k(\cdot,u) - k(\cdot,U)^T K(U,U)^{-1} k(u,U) \rangle_{H_k} \\
& \qquad = \langle k(\cdot,u) , k(\cdot,u)  \rangle_{H_k} - 2 \langle k(\cdot,u) , k(\cdot,U)^T K(U,U)^{-1} k(u,U)\rangle_{H_k} \\
& \qquad \qquad + \langle k(\cdot,U)^T K(U,U)^{-1} k(u,U),  k(\cdot,U)^T K(U,U)^{-1} k(u,U) \rangle_{H_k} \\
& \qquad = k(u,u) - 2 k(u,U)^T K(U,U)^{-1} k(u,U) + k(u,U)^T K(U,U)^{-1} k(u,U) \\
&\qquad = k_N(u,u).
\end{align*}
{The identity which leads to the third term in the penultimate line
uses the fact that $\langle k(\cdot,u'),  k(\cdot,u) \rangle_{H_k} = k(u,u')$, 
for any $u,u' \in X.$ If $\ell(u)=K(U,U)^{-1}k(u,U)$ then
\begin{align*}
\langle k(\cdot,U)^T K(U,U)^{-1} k(u,U),  k(\cdot,U)^T K(U,U)^{-1} k(u,U) \rangle_{H_k}&=\sum_{j,k} \ell_j(u) \langle k(\cdot,u^j),k(\cdot,u^k)\rangle_{H_k}\ell_k(u)\\
&= \sum_{j,k} \ell_j(u) k(u^j,u^k)\ell_k(u)\\
&=\ell(u)^TK(U,U) \ell(u)\\
&= k(u,U)^T K(U,U)^{-1} k(u,U) 
\end{align*}
as required.}
This completes the proof.% of the first claim. The second and third claims follow immediately.
\end{proof}

{The second string of equalities, appearing in the middle part of
the proof Proposition \ref{prop:predvar_sup}, might appear counter-intuitive at first glance in that the left-most
quantity is a norm squared of quantities which scale like $k$, whilst the
right-most quantity scales like $k$ itself. However, the space $H_k$ itself 
depends on the kernel $k$, and scales inversely proportional to $k$,
explaining that the identity is indeed dimensionally correct.}

\begin{remark}{\em ({\em Exponential convergence for the Gaussian kernel})  The RKHS corresponding to the Gaussian kernel \eqref{eq:gauss_cov} is no longer isomorphic to a Sobolev space; it is contained in $H^\tau(X)$, for any $\tau < \infty$. For functions $f$ in this RKHS,  Gaussian process regression with the Gaussian kernel converges exponentially in the fill distance $h_U$. For more details, see \cite{wendland}.}
\end{remark}

\begin{remark}\label{rem:mean}{\em ({\em Regression with non-zero mean}) If in \eqref{eq:gp} we use a non-zero mean $m(\cdot)$, the formula for the predictive mean $m_N^f$ changes to
\begin{equation}\label{eq:pred_eq_mean}
m_N^f(u) = m(u) + k(u,U)^T K(U,U)^{-1} (f(U) - m(U)),
\end{equation}
where $m(U) := [m(u^1), \dots, m(u^N)] \in \R^N$. The predictive covariance $k_N(u,u')$ is as in \eqref{eq:pred_eq}. As in the case $m \equiv 0$, we have $m_N^f(u^n) = f(u^n)$, for $n=1, \dots, N$, and $m_N^f$ is an interpolant of $f$. If $m \in H_k$, then $m_N^f$ given by \eqref{eq:pred_eq_mean} is also in $H_k$, and the proof techniques in \cite{nww05,nww06} can be applied. The conclusions of Propositions \ref{prop:mean_conv} and \ref{prop:mean_conv_int} then hold, with the factor $\|f\|$ in the error bounds replaced by $\|f\| + \|m\|$.
}
\end{remark}

\section{Approximation of the Bayesian posterior distribution}\label{sec:gp_app}
In this {section}, we analyse the error introduced in the posterior distribution $\mu^y$ when we use a Gaussian process emulator to approximate the parameter-to-observation map $\mathcal G$ or the negative log-likelihood $\Phi$. The aim is to show convergence, in a suitable sense, of the approximate posterior distributions to the true posterior distribution as the number of observations $N$ tends to infinity. For a given approximation $\mu^{y,N}$ of the posterior distribution {$\mu^y$}, we will focus on bounding the Hellinger distance \cite{stuart10} between the two distributions, which is defined as
\[
\dhh(\mu^y, \mu^{y,N}) = \left( \frac{1}{2} \Large{\int}_{X} \left(\sqrt{\frac{d \mu^y}{d \mu_0}} - \sqrt{\frac{d \mu^{y,N}}{d \mu_0}} \right)^2 d \mu_0 \right)^{1/2}.
\]
As proven in \cite[Lemma 6.12 and 6.14]{ds15}, the Hellinger distance provides a bound for the Total Variation distance
\[
\dtv(\mu^y, \mu^{y,N}) = \frac{1}{2} \sup_{\|f\|_\infty \leq 1} \left| \EE_{\mu^y}(f) - \EE_{\mu^{y,N}}(f) \right| \leq \sqrt{2} \; \dhh(\mu^y, \mu^{y,N}),
\] 
and for $f \in L^2_{\mu^y}(X) \cap L^2_{\mu^{y,N}}(X)$, the Hellinger distance also provides a bound on the error in expected values
\[
\left| \EE_{\mu^y}(f) - \EE_{\mu^{y,N}}(f) \right| \leq 2 (\EE_{\mu^y}(f^2) + \EE_{\mu^{y,N}}(f^2))^{1/2} \; \dhh(\mu^y, \mu^{y,N}).
\]

Depending on how we make use of the predictive process $\mathcal G_N$ or $\Phi_N$ to approximate the Radon-Nikodym derivative $\frac{\mathrm d \mu^y}{\mathrm d \mu_0}$, we obtain different approximations to the posterior distribution $\mu^y$. We will distinguish between approximations based solely on the predictive mean, and approximations that make use of the full predictive process.

\subsection{Approximation based on the predictive mean}
Using simply the predictive mean of a Gaussian process emulator of the parameter-to-observation map $\mathcal G$ or the negative log-likelihood $\Phi$, we can define the approximations $\mu^{y,N, \mathcal G}_\mathrm{mean}$ and $\mu^{y,N, \Phi}_\mathrm{mean}$, given by
%The maybe simplest approximation to the posterior is obtained by replacing the forward map $\mathcal G(u)$ by the predictive mean $m_N^*(u)$. Since the predictive mean is a continuous function of $u$, this results in a posterior distribution $\mu^{y,N}_\mathrm{mean}$ satisfying
\begin{align*}%\label{eq:rad_nik_mean}
\frac{d\mu^{y,N, \mathcal G}_\mathrm{mean}}{d\mu_0}(u) &= \frac{1}{Z_{N, \mathcal G}^\mathrm{mean}} \exp\big(-\frac{1}{2 \sigma_\eta^2} \left\| y -  m^\mathcal G _N(u) \right\|^2\big), \\
Z_{N, \mathcal G}^\mathrm{mean} &= \EE_{\mu_0}\Big(\exp\big(-\frac{1}{2 \sigma_\eta^2} \left\| y -  m^\mathcal G _N(u) \right\|^2\big)\Big), \\
\frac{d\mu^{y,N, \Phi}_\mathrm{mean}}{d\mu_0}(u) &= \frac{1}{Z_{N, \Phi}^\mathrm{mean}} \exp\big(- m_N^\Phi(u)\big), \\
Z_{N, \Phi}^\mathrm{mean} &= \EE_{\mu_0}\Big(\exp\big(-m_N^\Phi(u)\big)\Big),
\end{align*}
{where $m_N^\mathcal G(u) = [m_N^{\mathcal G^1}(u), \dots, m_N^{\mathcal G^J}(u)] \in \R^J$.} We have the following {lemma concerning} the normalising constants $Z_{N, \mathcal G}^\mathrm{mean}$ and $Z_{N, \Phi}^\mathrm{mean}$, {which is followed
by the main Theorem \ref{thm:hell_mean} and Corollary \ref{cor:rate_mean}
concerning the
approximations $\mu^{y,N, \mathcal G}_\mathrm{mean}, \mu^{y,N, \Phi}_\mathrm{mean}.$} 

\begin{lemma}\label{thm:bound_zmean} Suppose $\sup_{u \in X} \| \mathcal G(u) -  m^\mathcal G _N(u) \|$ and $\sup_{u \in X} | \Phi(u) -  m^\Phi_N(u) |$ converge to 0 as $N$ tends to $\infty$, and assume $\sup_{u \in X}\|\mathcal G(u)\| \leq C_\mathcal G$. Then there exist positive constants $C_1$ and $C_2$, independent of $U$ and $N$, such that
\[
C_1 \leq Z_{N, \mathcal G}^\mathrm{mean} \leq 1 \qquad \text{and} \qquad {C_2^{-1}} \leq Z_{N, \Phi}^\mathrm{mean} \leq {C_2}.
\]
\end{lemma}
\begin{proof}
Let us first consider $Z_{N, \mathcal G}^\mathrm{mean}$. The upper bound follows from a straight forward calculation, since the potential $\frac{1}{2 \sigma_\eta^2} \left\| y -  m^\mathcal G _N(u) \right\|^2$ is non-negative:
\[
Z_{N, \mathcal G}^\mathrm{mean} = \EE_{\mu_0}\Big(\exp\big(-\frac{1}{2 \sigma_\eta^2} \left\| y -  m^\mathcal G _N(u) \right\|^2\big)\Big) \leq \EE_{\mu_0}(1) = 1.
\]
For the lower bound, we have
\[
Z_{N, \mathcal G}^\mathrm{mean} \geq \EE_{\mu_0} \Big(\exp\big(- \frac{1}{2 \sigma_\eta^2} \; \sup_{u \in X}  \left\| y -  m^\mathcal G _N(u) \right\|^2\big)\Big) = \exp\big(- \frac{1}{2 \sigma_\eta^2} \; \sup_{u \in X}  \left\| y -  m^\mathcal G _N(u) \right\|^2\big),
\]
since $\int_X \mu_0(\mathrm d u) = 1$. Using the triangle inequality, the assumption $\sup_{u \in X}\|\mathcal G(u)\| \leq C_\mathcal G$ and the fact that every convergent sequence is bounded, we have
\begin{equation}\label{eq:bound_lemzmean}
\sup_{u \in X} \left\| y -  m^\mathcal G _N(u) \right\|^2 \leq \sup_{u \in X}  \| y -  \mathcal G(u) \|^2 + \sup_{u \in X} \left\| \mathcal G(u) -  m^\mathcal G _N(u) \right\|^2 =: - \ln C_1,
\end{equation}
where $C_1$ is independent of $U$ and $N$.

The proof for $Z_{N, \Phi}^\mathrm{mean}$ is similar. For the upper bound, we use $\int_X \mu_0(\mathrm d u) = 1$ and the triangle inequality to derive
\[
Z_{N, \Phi}^\mathrm{mean} \leq \sup_{u\in X} \exp\big(-m_N^\Phi(u)\big) \leq \exp\big(\sup_{u \in X} |m_N^\Phi(u)|\big) \leq \exp\big( \sup_{u\in X} |\Phi(u)| + \sup_{u\in X} |\Phi(u) - m_N^\Phi(u)| \big).
\]
Since $\sup_{u \in X}|\Phi(u)|$ is bounded when $\sup_{u \in X}\|\mathcal G(u)\|$ is bounded, the fact that every convergent sequence is bounded again gives
\[
\sup_{u\in X} |\Phi(u)| + \sup_{u\in X} |\Phi(u) - m_N^\Phi(u)| =: - \ln C_2,
\]
for a constant $C_2$ independent of $U$ and $N$. For the lower bound, we note that since $\int_X \mu_0(\mathrm d u) = 1$,
\[
Z_{N, \Phi}^\mathrm{mean} \geq \EE_{\mu_0}\Big(\exp\big(-\sup_{u \in X} |m_N^\Phi(u)|\big)\Big) = \exp\big(-\sup_{u \in X} |m_N^\Phi(u)|\big) \geq C_2^{-1}.
\]
\end{proof}

{ We would like to point out here that the assumptions in Lemma \ref{thm:bound_zmean} can be relaxed to assuming that the sequences $\sup_{u \in X} \| \mathcal G(u) -  m^\mathcal G _N(u) \|$ and $\sup_{u \in X} | \Phi(u) -  m^\Phi_N(u) |$ are bounded, since this is sufficient to prove the result.}

{We may now prove the desired theorem and corollary concerning 
$\mu^{y,N,\mathcal G}_\mathrm{mean}$ and $\mu^{y,N, \Phi}_\mathrm{mean}.$}

\begin{theorem}\label{thm:hell_mean} Under the Assumptions of {Lemma} \ref{thm:bound_zmean}, there exist constants $C_1$ and $C_2$, independent of $U$ and $N$, such that
\begin{align*}
\dhh(\mu^y, \mu^{y,N,\mathcal G}_\mathrm{mean}) & \leq C_1 \left\|\mathcal G -  m^\mathcal G _N \right\|_{L^2_{\mu_0}(X; \R^J)},  \\
\text{and} \quad \dhh(\mu^y, \mu^{y,N, \Phi}_\mathrm{mean}) &\leq C_2 \left\|\Phi - m^\Phi _N \right\|_{L^2_{\mu_0}(X)}.
\end{align*}
\end{theorem}
\begin{proof}
Let us first consider $\mu^{y,N,\mathcal G}_\mathrm{mean}$. By definition of the Hellinger distance, we have
\begin{align*}
&2  \; \dhh^2(\mu^y, \mu^{y,N,\mathcal G}_\mathrm{mean}) = \int_X \left( \sqrt{\frac{d\mu^y}{d\mu_0}} - \sqrt{\frac{d\mu^{y,N,\mathcal G}_\mathrm{mean}}{d\mu_0}} \right)^2 \mu_0(\mathrm{d}u) \\
&\leq \frac{2}{Z} \int_X \left(\exp\big(-\frac{1}{4 \sigma_\eta^2} \left\| y -  \mathcal G (u) \right\|^2\big) - \exp\big(-\frac{1}{4 \sigma_\eta^2} \left\| y -  m^\mathcal G _N(u) \right\|^2\big)\right)^2 \mu_0(\mathrm{d}u) \\
&\qquad + 2 \, Z_{N,\mathcal G}^\mathrm{mean} \left(Z^{-1/2} - (Z_{N, \mathcal G}^\mathrm{mean})^{-1/2} \right)^2 \\
&=: I + II.
\end{align*}
For the first term, we use the local Lipschitz continuity of the exponential function, together with the equality $a^2 - b^2 = (a-b)(a+b)$ and the reverse triangle inequality to bound
\begin{align*}
\frac{Z}{2} \; I &= \int_X \left(\exp\big(-\frac{1}{4 \sigma_\eta^2} \left\| y -  \mathcal G (u) \right\|^2\big) - \exp\big(-\frac{1}{4 \sigma_\eta^2} \left\| y -  m^\mathcal G _N(u) \right\|^2\big)\right)^2 \mu_0(\mathrm{d}u) \\
&\leq \int_X \left( \frac{1}{2 \sigma_\eta^2} \left( \left\| y -  \mathcal G(u) \right\|^2 - \left\| y -  m^\mathcal G _N(u) \right\|^2  \right) \right)^2 \mu_0(\mathrm{d}u) \\
&= \int_X \frac{1}{4 \sigma_\eta^4} \left( \| y - \mathcal G(u)\| + \|y - m^\mathcal G _N(u)\| \right)^2 \left\|\mathcal G(u) -  m^\mathcal G _N(u) \right\|^2  \mu_0(\mathrm{d}u) \\
&\leq \frac{1}{4 \sigma_\eta^4} \sup_{u \in X} \left(  \| y - \mathcal G(u)\| + \|y - m^\mathcal G _N(u)\| \right)^2 \left\|\mathcal G(u) -  m^\mathcal G _N(u) \right\|_{L^2_{\mu_0}(X; \R^J)}^2 
\end{align*}
As in equation \eqref{eq:bound_lemzmean}, the first supremum can be bounded independently of $U$ and $N$, from which it follows that 
\[
I \leq C \left\|\mathcal G(u) -  m^\mathcal G _N(u) \right\|_{L^2_{\mu_0}(X; \R^J)}^2,
\]
for a constant $C$ independent of $U$ and $N$.
For the second term, a very similar argument, together with {Lemma} \ref{thm:bound_zmean} and Jensen's inequality, shows 
\begin{align*}
II &= 2 \, Z_{N,\mathcal G}^\mathrm{mean} \left(Z^{-1/2} - (Z_{N, \mathcal G}^\mathrm{mean})^{-1/2} \right)^2 \\
&\leq 2 \, Z_{N,\mathcal G}^\mathrm{mean} \max(Z^{-3},(Z_{N, \mathcal G}^\mathrm{mean})^{-3}) |Z - Z_{N, \mathcal G}^\mathrm{mean}|^2 \\
&= 2 \, Z_{N,\mathcal G}^\mathrm{mean} \max(Z^{-3},(Z_{N\mathcal G}^\mathrm{mean})^{-3}) \left(\int_X \exp\big(-\frac{1}{4 \sigma_\eta^2} \left\| y -  \mathcal G (u) \right\|^2\big) - \exp\big(-\frac{1}{4 \sigma_\eta^2} \left\| y -  m^\mathcal G _N(u) \right\|^2\big) \mu_0(\mathrm{d}u)\right)^2 \\
&\leq C \left\|\mathcal G(u) -  m^\mathcal G _N(u) \right\|_{L^2_{\mu_0}(X; \R^J)}^2,
\end{align*}
for a constant $C$ independent of $U$ and $N$. 

The proof for $\mu^{y,N, \Phi}_\mathrm{mean}$ is similar. We use an identical corresponding splitting of the Hellinger distance $\dhh(\mu^y, \mu^{y,N, \Phi}_\mathrm{mean}) \leq I + II$. Using the local Lipschitz continuity of the exponential function, 
we have
\begin{align*}
\frac{Z}{2} \; I &= \int_X \left(\exp\big(-\Phi(u)\big) - \exp\big(- m^\Phi _N(u)\big)\right)^2 \mu_0(\mathrm{d}u) \\
&\leq  (1 + \exp\big(\sup_{u \in X} |m_N^\Phi(u)|\big)) \left\|\Phi(u) - m^\Phi _N(u) \right\|_{L^2_{\mu_0}(X)}^2,
\end{align*}
where the first factor can be bounded independently of $U$ and $N$ as in Lemma \ref{thm:bound_zmean}.

Using {Lemma} \ref{thm:bound_zmean} and Jensen's inequality, we furthermore have 
\begin{align*}
II &\leq 2 \, Z_{N,\Phi}^\mathrm{mean} \max(Z^{-3},(Z_{N, \Phi}^\mathrm{mean})^{-3}) \left(\int_X \exp\big(-\Phi(u)\big) - \exp\big(- m^\Phi _N(u)\big) \mu_0(\mathrm{d}u)\right)^2 \\
&\leq C \left\|\Phi(u) - m^\Phi _N(u) \right\|_{L^2_{\mu_0}(X)}^2,
\end{align*}
for a constant $C$ independent of $U$ and $N$.
\end{proof}

We remark here that Theorem \ref{thm:hell_mean} does not make any assumptions on the predictive means $m_N^\mathcal G$ and $m_N^\Phi$ other than the requirement that $\sup_{u \in X} \| \mathcal G(u) -  m^\mathcal G _N(u) \|$ and $\sup_{u \in X} | \Phi(u) -  m^\Phi_N(u) |$ converge to 0 as $N$ tends to $\infty$. Whether the predictive means are defined as in \eqref{eq:pred_eq}, or are derived by alternative approaches to Gaussian process regression \cite{rasmussen_williams}, does not affect the conclusions of Theorem \ref{thm:hell_mean}.
Under Assumption \ref{ass:reg}, we can combine Theorem \ref{thm:hell_mean} with Proposition \ref{prop:mean_conv} (or Proposition \ref{prop:mean_conv_int}) 
{with $\beta=0$} to obtain error bounds in terms of the fill distance of the design points.

\begin{corollary}\label{cor:rate_mean} Suppose $m_N^\Phi$ and $m_N^{\mathcal G^j}$, $j=1,\dots,J$, are defined as in \eqref{eq:pred_eq}, with Mat\`ern kernel $k=k_{\nu,\lambda,\sigma_k^2}$. Suppose { Assumption A holds}, Assumption \ref{ass:reg} { holds with $s=\nu + K/2$,} and the assumptions of Proposition \ref{prop:mean_conv} and Theorem \ref{thm:hell_mean} are satisfied. Then there exist constants $C_1$ and $C_2$, independent of $U$ and $N$, such that
\begin{align*}
\dhh(\mu^y, \mu^{y,N,\mathcal G}_\mathrm{mean}) \leq C_1 h_U^{\nu + K/2}, \quad
\text{and} \quad \dhh(\mu^y, \mu^{y,N, \Phi}_\mathrm{mean}) \leq C_2 h_U^{\nu + K/2}.
\end{align*}
\end{corollary}

{If Assumption \ref{ass:reg} holds only for some $s < \nu + K/2$, an analogue of Corollary \ref{cor:rate_mean} can be proved using Proposition \ref{prop:mean_conv_int} with $\beta = 0$. As already discussed in section \ref{ssec:gp_sk}, translating convergence rates in terms of the fill distance $h_U$ into rates in terms of the number of points $N$ typically leads to a strong dependence on the input dimension $K$. For uniform tensor grids $U$, the rates of convergence in $N$ predicted by Corollary \ref{cor:rate_mean} are given in Table \ref{tbl:conv}.}

\subsection{Approximations based on the predictive process}
Alternative to the mean-based approximations considered in the previous section, we now consider approximations to the posterior distribution $\mu^y$ obtained using the full predictive processes $\mathcal G_N$ and $\Phi_N$. {In contrast to the mean, the full Gaussian processes also carry information about the uncertainty in the emulator due to only using a finite number of function evaluations to construct it.}

For the remainder of this section, we denote by $\nu^\mathcal G_N$ the distribution of $\mathcal G_N$ and by $\nu^\Phi_N$ the distribution of $\Phi_N$, {for $N \in \mathbb N \cup \{0\}$}. We note that since the process $\mathcal G_N$ consists of $J$ independent Gaussian processes $\mathcal G_N^j$, the measure $\nu^\mathcal G_N$ is a product measure, $\nu^\mathcal G_N  = \prod_{j=1}^J \nu^{\mathcal G^j}_N$. {$\Phi_N$ is a Gaussian process with mean $m_N^\Phi$ and covariance kernel $k_N$, and $\mathcal G_N^j$, for $j=1, \dots, J$, is a Gaussian process with mean $m_N^{\mathcal G^j}$ and covariance kernel $k_N$.} Replacing $\mathcal G$ by $\mathcal G_N$ in \eqref{eq:def_like}, we obtain the approximation $\mu^{y,N,\mathcal G}_\mathrm{sample}$ given by
\begin{equation*}\label{eq:rad_nik_sample}
\frac{d\mu^{y,N,\mathcal G}_\mathrm{sample}}{d\mu_0}(u) = \frac{1}{Z_{N, \mathcal G}^\mathrm{sample}} \exp\big(-\frac{1}{2 \sigma_\eta^2} \left\| y -  \mathcal G_N (u) \right\|^2\big),
\end{equation*}
where 
\[
Z_{N, \mathcal G}^\mathrm{sample}= \EE_{\mu_0}\Big(\exp\big(-\frac{1}{2 \sigma_\eta^2} \left\| y -  \mathcal G_N (u) \right\|^2\big)\Big).
\]
Similarly, we define for the predictive process $\Phi_N$ the approximation $\mu^{y,N,\Phi}_\mathrm{sample}$ by
\begin{align*}
\frac{d\mu^{y,N,\Phi}_\mathrm{sample}}{d\mu_0}(u) = \frac{1}{Z_{N, \Phi}^\mathrm{sample}} \exp\big(- \Phi_N(u)\big), \qquad Z_{N, \Phi}^\mathrm{sample} = \EE_{\mu_0}\Big(\exp\big(- \Phi_N(u)\big)\Big).
\end{align*}
The measures $\mu^{y,N,\mathcal G}_\mathrm{sample}$ and $\mu^{y,N,\Phi}_\mathrm{sample}$ are random approximations of the deterministic measure $\mu^y.$ { The uncertainty in the posterior distribution introduced in this way can be thought of representing the uncertainty in the emulator, which in applications can be large (or comparable) to the uncertainty present in the observations. A user may want to take this into account to "inflate" the variance of the posterior distribution.} 

Deterministic approximations of the posterior distribution $\mu^y$ can now be obtained by taking the expected value with respect to the predictive processes $\mathcal G_N$ and $\Phi_N$. This results in the marginal approximations
\begin{align*}
\frac{d\mu^{y,N,\mathcal G}_\mathrm{marginal}}{d\mu_0}(u) &= \frac{1}{\EE_{\nu_N^\mathcal G}(Z_{N, \mathcal G}^\mathrm{sample})} \EE_{\nu_N^\mathcal G}\Big(\exp\big(-\frac{1}{2 \sigma_\eta^2} \left\| y -  \mathcal G_N (u) \right\|^2\big)\Big), \\
\frac{d\mu^{y,N,\Phi}_\mathrm{marginal}}{d\mu_0}(u) &= \frac{1}{\EE_{\nu_N^\Phi} (Z_{N, \Phi}^\mathrm{sample})} \EE_{\nu_N^\Phi} \Big(\exp\big(-\Phi_N (u) \big)\Big). 
\end{align*}
Note that by Tonelli's Theorem {(\cite{rudin}, a version of Fubini's Theorem for non-negative integrands)}, the measures $\mu^{y,N,\mathcal G}_\mathrm{marginal}$ and $\mu^{y,N,\Phi}_\mathrm{marginal}$ are indeed probability measures. { It can be shown that the above approximation of the likelihood is optimal in the sense that it minimises the $L^2$-error \cite{sn16}. In contrast to the approximation based on only the mean of the emulator, this approximation also takes into account the uncertainty of the emulator, although only in an averaged sense. The likelihood in the marginal approximations $\mu^{y,N,\mathcal G}_\mathrm{marginal}$ and $\mu^{y,N,\Phi}_\mathrm{marginal}$ involves computing an expectation. Methods from the pseudo-marginal MCMC literature \cite{ar09} could be used within an MCMC method in this context.}

Before proving bounds on the error in the marginal approximations $\mu^{y,N,\mathcal G}_\mathrm{marginal}$ and $\mu^{y,N,\Phi}_\mathrm{marginal}$ in section \ref{ssec:gp_app_marg}, and the error in the random approximations $\mu^{y,N,\mathcal G}_\mathrm{sample}$ and $\mu^{y,N,\Phi}_\mathrm{sample}$ in section \ref{ssec:gp_app_rand}, we crucially prove boundedness of the normalising constants $Z_{N, \mathcal G}^\mathrm{sample}$ and $Z_{N, \Phi}^\mathrm{sample}$ in section \ref{ssec:gp_app_z}.

\subsubsection{Moment bounds on $Z_{N, \mathcal G}^\mathrm{sample}$ and $Z_{N, \Phi}^\mathrm{sample}$}\label{ssec:gp_app_z}
Firstly, we recall the following classical results from the theory of Gaussian measures on Banach spaces \cite{daprato_zabczyk,adler}.

\begin{proposition}\label{prop:fernique} {\em (Fernique's Theorem)} Let $E$ be a separable Banach space and $\nu$ a centred Gaussian measure on $(E, \mathcal B(E))$. If $\lambda, r >0$ are such that
\[
\log \left( \frac{1- \nu(f \in E : \|f\|_E \leq r)}{\nu(f \in E : \|f\|_E \leq r)}  \right) \leq -1 -32 \lambda r^2,
\] 
then
\[
\int_E \exp(\lambda \|f\|_E^2) \nu(\mathrm{d} f) \; \leq \; \exp(16 \lambda r^2) + \frac{e^2}{e^2-1.}
\]
\end{proposition}
\vspace{1.5ex}
{\begin{proposition}\label{prop:borell_tis} {\em (Borell-TIS Inequality\footnote{{The Borell-TIS inequality is named after the mathematicians Borell and Tsirelson, Ibragimov and Sudakov, who independently proved the result.}})} Let $f$ be a scalar, almost surely bounded Gaussian field on a compact domain $T \subseteq \R^K$, with zero mean $\EE(f(t)) = 0$ and bounded variance $0 < \sigma^2_f := \sup_{t \in T} \VV(f(t)) < \infty$. Then $\EE(\sup_{t \in T} f(t)) < \infty$, and for all $r > 0$,
\[
\mathbb P(\sup_{t \in T} f(t) - \EE(\sup_{t \in T} f(t)) > r ) \leq \exp(-r^2/2\sigma_f^2).
\]
\end{proposition}
\vspace{1.5ex}
\begin{proposition}\label{prop:sud_fern} {\em (Sudakov-Fernique Inequality)} Let $f$ and $g$ be scalar, almost surely bounded Gaussian fields on a compact domain $T \subseteq \R^K$. Suppose $\EE((f(t)-f(s))^2) \leq \EE((g(t)-g(s))^2)$ and $\EE(f(t)) = \EE(g(t))$, for all $s,t \in T$. Then
\[
\EE(\sup_{t \in T} f(t)) \leq \EE(\sup_{t \in T} g(t)). 
\]
\end{proposition}
}

{Using these results, we are now ready to prove bounds on moments of $Z_{N, \mathcal G}^\mathrm{sample}$ and $Z_{N, \Phi}^\mathrm{sample}$, similar to those proved in {Lemma} \ref{thm:bound_zmean}.
The reader interested purely in approximation results for the posterior
can simply read the statements of the following two lemmas, and then
proceed directly to subsections \ref{ssec:gp_app_marg} and \ref{ssec:gp_app_rand}.}

{Recall that, as in \eqref{eq:gp}, $\Phi_0$ and $\mathcal G^j_0$ denote the initial Gaussian process models for $\Phi$ and $\mathcal G^j$, respectively, and, as in \eqref{eq:gp_pred}, $\Phi_N$ and $\mathcal G^j_N$ denote the conditioned Gaussian process models for $\Phi$ and $\mathcal G^j$, respectively.}

\begin{lemma}\label{thm:bound_zsample} {Let $X \subseteq \R^K$ be compact.} Suppose $\sup_{u \in X} \left\| \mathcal G(u) -  m^\mathcal G _N(u) \right\|$, $\sup_{u \in X} \left| \Phi(u) -  m^\Phi _N(u) \right|$ {and $\sup_{u \in X} k_N(u,u)$} converge to 0 as $N$ tends to infinity, and assume $\sup_{u \in X}\|\mathcal G(u)\| \leq C_\mathcal G < \infty$. {Suppose the assumptions of the Sudakov-Fernique inequality hold, for $g=\Phi_0$ and $f=\Phi_N - m_N^{\Phi}$, and for $g=\mathcal G^j_0$ and $f=\mathcal G^j_N - m_N^{\mathcal G^j}$, for $j \in \{1,\dots,J\}.$} Then, for any $1 \leq p < \infty$, there exist positive constants $C_1$ and $C_2$, independent of $U$ and $N$, such that {for all $N$ sufficiently large}
\[
C_1^{-1} \leq \EE_{\nu_N^\mathcal G} \big((Z_{N,\mathcal G}^\mathrm{sample})^p\big) \leq 1, \qquad \text{and} \qquad 1 \leq \EE_{\nu_N^\mathcal G} \big((Z_{N,\mathcal G}^\mathrm{sample})^{-p}\big) \leq C_1.
\]
and
\[
C_2^{-1} \leq \EE_{\nu_N^\Phi} \big((Z_{N,\Phi}^\mathrm{sample})^p\big) \leq C_2, \qquad \text{and} \qquad  C_2^{-1} \leq  \EE_{\nu_N^\Phi} \big((Z_{N,\Phi}^\mathrm{sample})^{-p}\big) \leq C_2.
\]
\end{lemma}
\begin{proof}
We start with $Z_{N, \mathcal G}^\mathrm{sample}$. Since the potential $\frac{1}{2 \sigma_\eta^2} \left\| y -  \mathcal G_N (u) \right\|^2$ is non-negative and $\int_X \mu_0(\mathrm d u) = 1 = \int_{C^0(X; \R^J)} \nu^\mathcal G_N (\rm{d} \mathcal G_N)$, we have for any $1 \leq p < \infty$,
\[
\EE_{\nu_N^\mathcal G}((Z_{N, \mathcal G}^\mathrm{sample})^p) = \int_{C^0(X; \R^J)} \left(\int_X \exp\big(- \frac{1}{2 \sigma_\eta^2} \left\| y -  \mathcal G_N (u \right\|^2\big) \mu_0(\rm{d} u) \right)^p \nu^\mathcal G_N (\rm{d} \mathcal G_N) \leq 1.
\]
From Jensen's inequality, it then follows that
\[
\EE_{\nu_N^\mathcal G}((Z_{N, \mathcal G}^\mathrm{sample})^{-p}\big) \geq \big(\EE_{\nu_N^\mathcal G}((Z_{N, \mathcal G}^\mathrm{sample})^p)\big)^{-1} \geq 1.
\]
To determine $C_1$, we use the triangle inequality to bound, for any $1 \leq p < \infty$,
\begin{align*}
&\EE_{\nu_N^\mathcal G}\big((Z_{N,\mathcal G}^\mathrm{sample})^{-p}\big) = \int_{C^0(X; \R^J)} \left(\int_X \exp\big(- \frac{1}{2 \sigma_\eta^2} \left\| y -  \mathcal G_N(u) \right\|^2\big) \mu_0(\rm{d} u)\right)^{-p} \nu_N^\mathcal{G} (\rm{d} \mathcal G_N) \\
&\leq \int_{C^0(X; \R^J)} \big(\exp\big(-\frac{1}{2 \sigma_\eta^2} \sup_{u \in X} \left\| y -  \mathcal G_N(u) \right\|^2\big)\big)^{-p} \nu_N^\mathcal G (\rm{d} \mathcal G_N) \\
&= \int_{C^0(X; \R^J)} \exp\big(\frac{p}{2 \sigma_\eta^2} \sup_{u \in X} \left\| y -  \mathcal G_N(u) \right\|^2\big) \nu_N^\mathcal G (\rm{d} \mathcal G_N) \\
&\leq \exp\left(\frac{\sup_{u \in X} \|y - m_N^\mathcal G(u)\|^2}{2 p^{-1} \sigma_\eta^2} \right) \int_{C^0(X; \R^J)} \exp\left(\frac{\sup_{u \in X} \|\mathcal G_N(u) - m_N^\mathcal G(u)\|^2}{2 p^{-1} \sigma_\eta^2} \right) \nu^\mathcal G_N (\rm{d} \mathcal G_N).
\end{align*}
The first factor {can} be bounded independently of $U$ and $N$ using the triangle inequality, together with $\sup_{u \in X}\|\mathcal G(u)\| \leq C_\mathcal G$ and $\sup_{u \in X} \left\| \mathcal G(u) -  m^\mathcal G _N(u) \right\| \rightarrow 0$ as $N \rightarrow \infty$. For the second factor, we use Fernique's Theorem (Proposition \ref{prop:fernique}). First, we note that {(using independence)}
\begin{align*}
&\int_{C^0(X; \R^J)} \exp\left(\frac{\sup_{u \in X} \|\mathcal G_N(u) - m_N^\mathcal G(u)\|^2}{2 p^{-1} \sigma_\eta^2} \right) \nu^\mathcal G_N (\rm{d} \mathcal G_N) \\
&= \int_{C^0(X; \R^J)} \exp\left(\frac{\sup_{u \in X} \sum_{j=1}^J |{\mathcal G^j}_N(u) - m_N^{\mathcal G^j}(u)|^2}{2 p^{-1} \sigma_\eta^2} \right) \nu^\mathcal G_N (\rm{d} \mathcal G_N) \\
&\leq \int_{C^0(X; \R^J)} \exp\left(\sum_{j=1}^J  \frac{\sup_{u \in X} |{\mathcal G^j}_N(u) - m_N^{\mathcal G^j}(u)|^2}{2 p^{-1} \sigma_\eta^2} \right) \nu^\mathcal G_N (\rm{d} \mathcal G_N) \\
&= \int_{C^0(X; \R^J)} \prod_{j=1}^J \exp\left( \frac{\sup_{u \in X} |{\mathcal G^j}_N(u) - m_N^{\mathcal G^j}(u)|^2}{2 p^{-1} \sigma_\eta^2} \right) \nu^\mathcal G_N (\rm{d} \mathcal G_N) \\
&= \prod_{j=1}^J \int_{C^0(X)} \exp\left( \frac{\sup_{u \in X} |{\mathcal G^j}_N(u) - m_N^{\mathcal G^j}(u)|^2}{2 p^{-1} \sigma_\eta^2} \right) \nu^{\mathcal G^j}_N (\rm{d} {\mathcal G^j_N}).
\end{align*}
{It remains to show that, for $N$ sufficiently large, the assumptions of Fernique's Theorem hold for $\lambda = p \sigma_\eta^{-2}/2$ and a value of $r$ independent of $U$ and $N$, for $\nu$ equal to the push-forward of $\nu_N^{\mathcal G^j}$ under the map $T(f) = f - m_N^{\mathcal G^j}$.
Denote by $B^{\mathcal G^j}_{N,r} \subset C^0(X)$ the set of all functions $f$ such that $\|f - m^{\mathcal G^j}_N\|_{C^0(X)} \leq r$, for some fixed $r > 0$ and $1 \leq j \leq J$. Let $\overline{\mathcal G^j_N} = \mathcal G^j_N - m^{\mathcal G^j}_N$. By the Borell-TIS Inequality, we have for all $r > \EE(\sup_{u \in X} (\overline{\mathcal G^j_N}(u) )$,
\begin{equation*}%\label{eq:bound1}
\nu_N^{\mathcal G^j}(\mathcal G^j_N : \sup_{u \in X} \overline{\mathcal G^j_N}(u) > r) \leq \exp\left(-\frac{\big(r - \EE(\sup_{u \in X} \overline{\mathcal G^j_N}(u)\big)^2}{2 \sigma_{N}^2} \right),
\end{equation*}
where $\sigma_{N}^2 := \sup_{u \in X} k_N(u,u)$. By assumption,  $\EE_{\nu_N^j}((\overline{\mathcal G^j_N}(u) - \overline{\mathcal G^j_N}(u'))^2) \leq \EE_{\nu_0^j}((\mathcal G^j_0(u) - \mathcal G^j_0(u'))^2)$, and so $\EE(\sup_{u \in X} (\overline{\mathcal G^j_N}(u) ) \leq \EE(\sup_{u \in X} (\mathcal G^j_0(u) )$, by the Sudakov-Fernique Inequality. We can hence choose $r > \EE(\sup_{u \in X} (\mathcal G^j_0(u) )$, independent of $U$ and $N$, such that the bound 
\begin{equation*}%\label{eq:bound2}
\nu_N^{\mathcal G^j}(\mathcal G^j_N : \sup_{u \in X} \overline{\mathcal G^j_N}(u) > r) \leq \exp\left(-\frac{\big(r - \EE(\sup_{u \in X} \mathcal G^j_0(u)\big)^2}{2 \sigma_{N}^2} \right),
\end{equation*}
holds for all $N \in \N$.
By assumption we have $\sigma_{N}^2 \rightarrow 0$ as $N \rightarrow \infty$, and by the symmetry of Gaussian measures, we hence have $\nu_N^{\mathcal G^j}(B^{\mathcal G^j}_{N,r}) \rightarrow 1$ as $N \rightarrow \infty$, for all $r > \EE(\sup_{u \in X} (\mathcal G^j_0(u) )$. For $N = N(p)$ sufficiently large, the inequality
\[
\log \left( \frac{1- \nu_N^{\mathcal G^j}(B^{\mathcal G^j}_{N,r})}{\nu_N^{\mathcal G^j}(B^{\mathcal G^j}_{N,r})}  \right) \leq -1 -32 \lambda r^2,
\] 
in the assumptions of Fernique's Theorem is then satisfied, for $\lambda = p \sigma_\eta^{-2}/2$ and $r > \EE(\sup_{u \in X} (\mathcal G^j_0(u) )$, both independent of $U$ and $N$. Hence, $\EE_{}\big((Z_{N, \mathcal G}^\mathrm{sample})^{-p}\big) \leq C_1(p)$, for a constant $C_1(p) < \infty$ independent of $U$ and $N$.}
From Jensen's inequality, it then finally follows that
\[
\EE_{\nu_N^\mathcal G}((Z_{N, \mathcal G}^\mathrm{sample})^{p}\big) \geq \big(\EE_{\nu_N^\mathcal G}((Z_{N, \mathcal G}^\mathrm{sample})^{-p})\big)^{-1} \geq C_1^{-1}(p).
\]

The proof for $Z_{N, \Phi}^\mathrm{sample}$ is similar.  
Using $\int_X \mu_0(\mathrm d u) = 1$ and the triangle inequality, we have
\begin{align*}
\EE_{\nu_N^\Phi}\big((Z_{N,\Phi}^\mathrm{sample})^p\big) &= \int_{C^0(X)} \left(\int_X \exp\big(- \Phi_N(u)\big) \mu_0(\rm{d} u) \right)^p \nu^\Phi_N (\rm{d} \Phi_N) \\
&\leq \int_{C^0(X)} \exp\big( p \sup_{u \in X} |\Phi_N(u)| \big) \nu^\Phi_N (\rm{d} \Phi_N) \\
&\leq \exp\big( p \sup_{u \in X} |m_N^\Phi(u)| \big) \int_{C^0(X)} \exp\big( p \sup_{u \in X} |\Phi_N(u) - m_N^\Phi(u)| \big) \nu^\Phi_N (\rm{d} \Phi_N).
\end{align*}
The first factor can be bounded independently of $U$ and $N$, since $\sup_{u \in X}\|\mathcal G(u)\| \leq C_\mathcal G$ and $\sup_{u \in X} \left| \Phi(u) -  m^\Phi _N(u) \right|$ converges to 0 as $N \rightarrow \infty$. {The second factor can be bounded by Fernique's Theorem. Using the same proof technique as above, we can show that $\nu^\Phi_N(B^\Phi_{N,r}) \rightarrow 1$ as $N \rightarrow \infty$ for all $r > \EE(\sup_{u \in X} (\Phi_0(u) )$, where $B^\Phi_{N,r} \subset C^0(X)$ denotes the set of all functions $f$ such that $\|f - m^\Phi_N\|_{C^0(X)} \leq r$.
Hence, it is possible to choose $r >0$, independent of $U$ and $N$, such that the assumptions of Fernique's Theorem hold for $\nu$ equal to the push-forward of $\nu_N^\Phi$ under the map $T(f) = f - m_N^\Phi$, for some $\lambda > 0$ also independent of $U$ and $N$. By Young's inequality, we have
\[
\exp\big( p \sup_{u \in X} |\Phi_N(u) - m_N^\Phi(u)| \big) \leq  \exp\big( \lambda \sup_{u \in X} |\Phi_N(u) - m_N^\Phi(u)|^2 + p^2/4 \lambda \big), 
\] 
and it follows that $\EE_\omega\big((Z_{N,\Phi}^\mathrm{sample})^p\big) \leq C_2(p)$, for a constant $C_2(p) < \infty$ independent of $U$ and $N$, for any $1 \leq p < \infty$.}
Furthermore, we note
\begin{align*}
\EE_{\nu_N^\Phi}\big((Z_{N,\Phi}^\mathrm{sample})^{-p}\big) \leq \int_{C^0(X; \R)} \exp\big(p \sup_{u \in X} |\Phi_N(u)|\big) \nu_N^\Phi (\rm{d} \Phi_N) \leq C_2(p).
\end{align*}
By Jensen's inequality, we finally have $\EE_{\nu_N^\Phi}\big((Z_{N,\Phi}^\mathrm{sample})^{-p}\big) \geq C_2(p)^{-1} $ and $\EE_{\nu_N^\Phi} \big((Z_{N,\Phi}^\mathrm{sample})^p\big) \geq C_2(p)^{-1}$.
\end{proof}

{ We would like to point out here that the assumption that $\sup_{u \in X} k_N(u,u)$ converges to 0 as $N$ tends to infinity in Lemma \ref{thm:bound_zsample} is crucial in order to enable the choice of any $1 \leq p < \infty$. This is related to the fact that the parameter $\lambda$ needs to be sufficiently small compared to $\sup_{u \in X} k_N(u,u)$ in order to satisfy the assumptions of Fernique's Theorem.}

{In {Lemma} \ref{thm:bound_zsample}, we supposed that the assumptions of the Sudakov-Fernique inequality hold, for $g=\Phi_0$ and $f=\Phi_N - m_N^{\Phi}$, and for $g=\mathcal G^j_0$ and $f=\mathcal G^j_N - m_N^{\mathcal G^j}$, for $j \in \{1,\dots,J\}$. This is an assumption on the predictive variance $k_N$. In the following Lemma, we prove this assumption for the predictive variance given in \eqref{eq:pred_eq}.}

{\begin{lemma}\label{lem:predvar_lip} Suppose the predictive variance $k_N$ is given by \eqref{eq:pred_eq}. Then the assumptions of the Sudakov-Fernique inequality hold, for $g=\Phi_0$ and $f=\Phi_N - m_N^{\Phi}$, and for $g=\mathcal G^j_0$ and $f=\mathcal G^j_N - m_N^{\mathcal G^j}$, for $j \in \{1,\dots,J\}$.
\end{lemma}
\begin{proof}
We give a proof for $g=\Phi_0$ and $f=\Phi_N - m_N^{\Phi}$, the proof for $g=\mathcal G^j_0$ and $f=\mathcal G^j_N - m_N^{\mathcal G^j}$, for $j \in \{1,\dots,J\}$, is identical. 
For any $u, u' \in X$, we have $\EE_{\nu^{\Phi}_0} (\Phi_0(u)) = 0 = \EE_{\nu^{\Phi}_N}  (\Phi_N(u) - m_N^{\Phi}(u))$, and  
\begin{align*}
\EE_{\nu^{\Phi}_N} \left( ((\Phi_N(u) - m^{\Phi} _N(u)) - (\Phi_N(u') - m^{\Phi} _N(u')))^2\right) &= k_N(u,u) - k_N(u,u') - k_N(u',u) + k_N(u',u'), \\
\EE_{\nu^{\Phi}_0} \left( (\Phi_0(u)  - \Phi_0(u'))^2\right) &= k(u,u) - k(u,u') - k(u',u) + k(u',u').
\end{align*}
By \eqref{eq:pred_eq}, we have
\[
k_N(u,u') = k(u,u') - k(u,U)^T \; K(U,U)^{-1} \; k(u',U),
\]
and so
\begin{align*}
&\EE_{\nu^{\Phi}_0} \left( (\Phi_0(u)  - \Phi_0(u'))^2\right) - \EE_{\nu^{\Phi}_N} \left( ((\Phi_N(u) - m^{\Phi} _N(u)) - (\Phi_N(u') - m^{\Phi} _N(u')))^2\right) \\
&= \big( k(u,U)^T - k(u',U)^T \big)\; K(U,U)^{-1} \; \big(k(u,U) - k(u',U)\big) \\
&\geq 0,
\end{align*}
since the matrix $K(U,U)^{-1}$ is positive definite.
\end{proof}}

We are now ready to prove bounds on the approximation error in the posterior distributions.

\subsubsection{Error in the marginal approximations $\mu^{y,N,\mathcal G}_\mathrm{marginal}$ and $\mu^{y,N,\Phi}_\mathrm{marginal}$}\label{ssec:gp_app_marg}

We start by analysing the error in the marginal approximations $\mu^{y,N,\mathcal G}_\mathrm{marginal}$ and $\mu^{y,N,\Phi}_\mathrm{marginal}$.

\begin{theorem}\label{thm:hell_marginal} Under the assumptions of {Lemma} \ref{thm:bound_zsample}, there exist constants $C_1$ and $C_2$, independent of $U$ and $N$, such that for any $\delta > 0$,
\begin{align*}
\dhh(\mu^y, \mu^{y,N,\mathcal G}_\mathrm{marginal}) &\leq C_1 \left\|\Big(\EE_{\nu_N^\mathcal G} \Big(\|\mathcal G  - \mathcal G_N \|^{1 + \delta} \Big)\Big)^{1/(1+\delta)}\right\|_{L^2_{\mu_0}(X)},  \\
\dhh(\mu^y, \mu^{y,N,\Phi}_\mathrm{marginal}) &\leq C_2 \left\|\EE_{\nu_N^\Phi} \left( |\Phi - \Phi_N| ^{1 + \delta} \right)^{1/(1+\delta)} \right\|_{L^2_{\mu_0}(X)}.
\end{align*}
\end{theorem}
\begin{proof}
We start with $\mu^{y,N,\mathcal G}_\mathrm{marginal}$. By the definition of the Hellinger distance, we have
\begin{align*}
&2  \; \dhh^2(\mu^y, \mu^{y,N,\mathcal G}_\mathrm{marginal}) = \int_X \left( \sqrt{\frac{d\mu^y}{d\mu_0}} - \sqrt{\frac{d\mu^{y,N,\mathcal G}_\mathrm{marginal}}{d\mu_0}} \right)^2 \mu_0(\mathrm{d}u) \\
&\leq \frac{2}{Z} \int_X \left(\sqrt{\exp\big(-\frac{1}{2 \sigma_\eta^2} \left\| y -  \mathcal G (u) \right\|^2\big)} - \sqrt{\EE_{\nu_N^\mathcal G} \Big(\exp\big(-\frac{1}{2 \sigma_\eta^2} \left\| y -  \mathcal G_N (u) \right\|^2\big)\Big)}\right)^2 \mu_0(\mathrm{d}u) \\
& \qquad + 2 \EE_{\nu_N^\mathcal G} \big(Z_{N,\mathcal G}^\mathrm{sample}\big) \left(Z^{-1/2} - \EE_{\nu_N^\mathcal G} \big(Z_{N,\mathcal G}^\mathrm{sample}\big)^{-1/2} \right)^2 \\
&= I + II.
\end{align*}

For the first term, we use the (in)equalities $a-b=(a^2-b^2)/(a+b)$ and $(\sqrt{a}+\sqrt{b})^2 \geq a + b$, for $a,b>0$, to derive
\begin{align*}
\frac{Z}{2} I &= \int_X \left(\sqrt{\exp\big(-\frac{1}{2 \sigma_\eta^2} \left\| y -  \mathcal G (u) \right\|^2\big)} - \sqrt{\EE_{\nu_N^\mathcal G} \Big(\exp\big(-\frac{1}{2 \sigma_\eta^2} \left\| y -  \mathcal G_N (u) \right\|^2\big)\Big)}\right)^2 \mu_0(\mathrm{d}u) \\
&\leq \int_X \frac{ \left( \exp\big(-\frac{1}{2 \sigma_\eta^2} \left\| y -  \mathcal G (u) \right\|^2\big) - \EE_{\nu_N^\mathcal G} \Big(\exp\big(-\frac{1}{2 \sigma_\eta^2} \left\| y -  \mathcal G_N (u) \right\|^2\big)\Big) \right)^2 }{\exp\big(-\frac{1}{2 \sigma_\eta^2} \left\| y -  \mathcal G (u) \right\|^2\big) + \EE_{\nu_N^\mathcal G} \Big(\exp\big(-\frac{1}{2 \sigma_\eta^2} \left\| y -  \mathcal G_N (u) \right\|^2\big)\Big)} \mu_0(\mathrm{d}u) \\
&\leq \sup_{u \in X} \left(\exp\big(-\frac{1}{2 \sigma_\eta^2} \left\| y -  \mathcal G (u) \right\|^2\big) + \EE_{\nu_N^\mathcal G} \Big(\exp\big(-\frac{1}{2 \sigma_\eta^2} \left\| y -  \mathcal G_N (u) \right\|^2\big)\Big)\right)^{-1} \\
&\qquad \qquad \int_X \left( \exp\big(-\frac{1}{2 \sigma_\eta^2} \left\| y -  \mathcal G (u) \right\|^2\big) - \EE_{\nu_N^\mathcal G} \Big(\exp\big(-\frac{1}{2 \sigma_\eta^2} \left\| y -  \mathcal G_N (u) \right\|^2\big)\Big) \right)^2 \mu_0(\mathrm{d}u).
\end{align*}
For the first factor, using the convexity of $1/x$ on $(0,\infty)$, together with Jensen's inequality, we have for all $u \in X$ the bound
\begin{align*}
&\left(\exp\big(-\frac{1}{2 \sigma_\eta^2} \left\| y -  \mathcal G (u) \right\|^2\big) + \EE_{\nu_N^\mathcal G} \Big(\exp\big(-\frac{1}{2 \sigma_\eta^2} \left\| y -  \mathcal G_N (u) \right\|^2\big)\Big)\right)^{-1} \\
& \qquad \leq \exp\big(-\frac{1}{2 \sigma_\eta^2} \left\| y -  \mathcal G (u) \right\|^2\big)^{-1} + \EE_{\nu_N^\mathcal G} \Big(\exp\big(-\frac{1}{2 \sigma_\eta^2} \left\| y -  \mathcal G_N (u) \right\|^2\big)\Big)^{-1} \\
& \qquad \leq \exp\big(\frac{1}{2 \sigma_\eta^2} \left\| y -  \mathcal G (u) \right\|^2\big) + \EE_{\nu_N^\mathcal G} \Big(\exp\big(\frac{1}{2 \sigma_\eta^2} \left\| y -  \mathcal G_N (u) \right\|^2\big)\Big) \\
& \qquad \leq \exp\big(\frac{1}{2 \sigma_\eta^2} \sup_{u \in X} \left\| y -  \mathcal G (u) \right\|^2\big) + \EE_{\nu_N^\mathcal G} \Big(\exp\big(\frac{1}{2 \sigma_\eta^2} \sup_{u \in X} \left\| y -  \mathcal G_N (u) \right\|^2\big)\Big).
\end{align*}
As in the proof of {Lemma} \ref{thm:bound_zsample}, it then follows by Fernique's Theorem that the right hand side can be bounded by a constant independent of $U$ and $N$. 

For the second factor in the bound on $\frac{Z}{2} I$, the linearity of expectation, the local Lipschitz continuity of the exponential function, the equality $a^2-b^2 = (a-b)(a+b)$, the reverse triangle inequality and H\"older's inequality with conjugate exponents $p=(1+\delta)/\delta$ and $q=1+\delta$ give
\begin{align*}
&\int_X \left( \exp\big(-\frac{1}{2 \sigma_\eta^2} \left\| y -  \mathcal G (u) \right\|^2\big) - \EE_{\nu_N^\mathcal G} \Big(\exp\big(-\frac{1}{2 \sigma_\eta^2} \left\| y -  \mathcal G_N (u) \right\|^2\big)\Big) \right)^2 \mu_0(\mathrm{d}u)\\
&= \int_X \left( \EE_{\nu_N^\mathcal G} \Big( \exp\big(-\frac{1}{2 \sigma_\eta^2} \left\| y -  \mathcal G (u) \right\|^2\big) - \exp\big(-\frac{1}{2 \sigma_\eta^2} \left\| y -  \mathcal G_N (u) \right\|^2\big)\Big) \right)^2 \mu_0(\mathrm{d}u)\\
&\leq 2 \int_X \left( \EE_{\nu_N^\mathcal G} \Big( |\frac{1}{2 \sigma_\eta^2} \left\| y -  \mathcal G (u) \right\|^2 - \frac{1}{2 \sigma_\eta^2} \left\| y -  \mathcal G_N (u) \right\|^2|\Big) \right)^2 \mu_0(\mathrm{d}u)\\
&\leq \frac{2}{4 \sigma_\eta^4} \int_X  \left(\EE_{\nu_N^\mathcal G} \Big( \left(\left\| y -  \mathcal G (u) \right\| + \left\| y -  \mathcal G_N (u) \right\|\right) \|\mathcal G (u) - \mathcal G_N (u)\| \Big) \right)^2 \mu_0(\mathrm{d}u)\\
&\leq \frac{2}{4 \sigma_\eta^4} \int_X  \Big(\EE_{\nu_N^\mathcal G} \Big( \left(\left\| y -  \mathcal G (u) \right\| + \left\| y -  \mathcal G_N (u) \right\|\right)^{(1+\delta)/\delta} \Big) \Big)^{2 \delta / (1+ \delta)} \Big(\EE_{\nu_N^\mathcal G} \Big(\|\mathcal G (u) - \mathcal G_N (u)\|^{1 + \delta} \Big) \Big)^{2/(1+\delta)} \mu_0(\mathrm{d}u) \\
&\leq \frac{2}{4 \sigma_\eta^4} \sup_{u \in X} \Big(\EE_{\nu_N^\mathcal G} \Big( \left(\left\| y -  \mathcal G (u) \right\| + \left\| y -  \mathcal G_N (u) \right\|\right)^{(1+\delta)/\delta} \Big) \Big)^{2 \delta / (1+ \delta)} \int_X \Big( \EE_{\nu_N^\mathcal G} \Big(\|\mathcal G (u) - \mathcal G_N (u)\|^{1 + \delta} \Big) \Big)^{2/(1+\delta)} \mu_0(\mathrm{d}u),
\end{align*}
for any $\delta > 0$. The supremum in the above expression can be bounded by a constant independent of $U$ and $N$ by Fernique's Theorem as in the proof of 
{Lemma} \ref{thm:bound_zsample}, since $\sup_{u \in X}\|\mathcal G(u)\| \leq C_\mathcal G < \infty$. It follows that there exists a constant $C$ independent of $U$ and $N$ such that
\[
\frac{Z}{2} I \leq C \left\|\Big( \EE_{\nu_N^\mathcal G} \Big(\|\mathcal G_N - \mathcal G\|^{1 + \delta} \Big) \Big)^{1/(1+\delta)}\right\|^2_{L^2_{\mu_0}(X)}.
\]

For the second term in the bound on the Hellinger distance, we have
\begin{equation*}
\frac{1}{2 \EE_{\nu_N^\mathcal G}\big(Z_{N,\mathcal G}^\mathrm{sample}\big)} II =  \left(Z^{-1/2} - \big(\EE_{\nu_N^\mathcal G} \big(Z_{N,\mathcal G}^\mathrm{sample}\big)\big)^{-1/2} \right)^2  \leq \max(Z^{-3},(\EE_{\nu_N^\mathcal G} \big(Z_{N,\mathcal G}^\mathrm{sample}\big))^{-3}) |Z - \EE_{\nu_N^\mathcal G} \big(Z_{N,\mathcal G}^\mathrm{sample}\big)|^2.
\end{equation*}
Using the linearity of expectation, Tonelli's Theorem and Jensen's inequality, we have
\begin{align*}
&\left|Z - \EE_{\nu_N^\mathcal G} \big(Z_{N,\mathcal G}^\mathrm{sample}\big)\right|^2 \\
&= \left| \int_X  \EE_{\nu_N^\mathcal G} \Big( \exp\big(-\frac{1}{2 \sigma_\eta^2} \left\| y -  \mathcal G (u) \right\|^2\big) - \exp\big(-\frac{1}{2 \sigma_\eta^2} \left\| y -  \mathcal G_N (u) \right\|^2\big)\Big) \mu_0(\mathrm{d}u) \right|^2 \\
&\leq \int_X \left( \EE_{\nu_N^\mathcal G} \Big( \exp\big(-\frac{1}{2 \sigma_\eta^2} \left\| y -  \mathcal G (u) \right\|^2\big) - \exp\big(-\frac{1}{2 \sigma_\eta^2} \left\| y -  \mathcal G_N (u) \right\|^2\big)\Big) \right)^2 \mu_0(\mathrm{d}u).
\end{align*}
which can now be bounded as before. The first claim of the theorem now follows by {Lemma} \ref{thm:bound_zsample}.

The proof for $\mu^{y,N,\Phi}_\mathrm{marginal}$ is similar. We use an identical corresponding splitting of the Hellinger distance $\dhh(\mu^y, \mu^{y,N,\Phi}_\mathrm{marginal}) \leq I + II$. For the first term, we have
%use the (in)equalities $a-b=(a^2-b^2)/(a+b)$ and $(\sqrt{a}+\sqrt{b})^2 \geq a + b$, for $a,b>0$, to derive
\begin{align*}
\frac{Z}{2} I &= \int_X \left(\sqrt{\exp\big(-\Phi(u)\big)} - \sqrt{\EE_{\nu_N^\Phi} \Big(\exp\big(-\Phi_N(u) \big)\Big)}\right)^2 \mu_0(\mathrm{d}u) \\
&\leq \sup_{u \in X} \left(\exp\big(-\Phi(u)\big) + \EE_{\nu_N^\Phi} \Big(\exp\big(-\Phi_N(u)\big)\Big)\right)^{-1} \\
&\qquad \qquad \int_X \left( \exp\big(-\Phi(u)\big) - \EE_{\nu_N^\Phi} \Big(\exp\big(-\Phi_N(u)\big)\Big) \right)^2 \mu_0(\mathrm{d}u).
\end{align*}
The first factor can again be bounded using Jensen's inequality,
%For the first factor, using the convexity of $1/x$ on $(0,\infty)$, together with Jensen's inequality, we have for all $u \in X$ the bound
\begin{align*}
\sup_{u \in X}  \left(\exp\big(-\Phi(u)\big) + \EE_{\nu_N^\Phi} \Big(\exp\big(-\Phi_N(u)\big)\Big)\right)^{-1} \leq \exp\big(\sup_{u \in X} \Phi(u) \big) + \EE_{\nu_N^\Phi} \Big(\exp\big( \sup_{u \in X} \Phi_N(u)\big)\Big),
\end{align*}
which as in the proof of {Lemma} \ref{thm:bound_zsample}, can be bounded by a constant independent of $U$ and $N$ by Fernique's Theorem. 

For the second factor in the bound on $\frac{Z}{2} I$, the linearity of expectation, H\"older's inequality with conjugate exponents $p=(1+\delta)/\delta$ and $q=1+\delta$ and the inequality $|\exp(x) - \exp(y)| \leq (\exp(x) + \exp(y)) |x-y|$, for all $x,y \in \R$, give
\begin{align*}
&\int_X \left( \exp\big(-\Phi(u)\big) - \EE_{\nu_N^\Phi} \Big(\exp\big(-\Phi_N(u)\big)\Big) \right)^2 \mu_0(\mathrm{d}u)\\
&= \int_X \left( \EE_{\nu_N^\Phi} \Big( \exp\big(-\Phi(u)\big) - \exp\big(-\Phi_N(u)\big)\Big) \right)^2 \mu_0(\mathrm{d}u)\\
&\leq \int_X \left( \EE_{\nu_N^\Phi} \Big( (\exp\big(-\Phi(u)\big) + \exp\big(-\Phi_N(u))\big)^{(1+\delta)/\delta} \Big)^{\delta/(1+\delta)} \EE_{\nu_N^\Phi} \Big( |\Phi(u) - \Phi_N(u)|^{1+\delta} \Big)^{1/(1+\delta)} \right)^2 \mu_0(\mathrm{d}u)\\
&\leq \sup_{u \in X} \EE_{\nu_N^\Phi} \Big( (\exp\big(-\Phi(u)\big) + \exp\big(-\Phi_N(u))\big)^{(1+\delta)/\delta} \Big)^{2\delta/(1+\delta)} \left\|\EE_{\nu_N^\Phi} \left( |\Phi(u) - \Phi_N(u)|^{1+\delta} \right)^{1/(1+\delta)} \right\|^2_{L^2_{\mu_0}(X)},
\end{align*}
where the first factor can be bounded independently of $U$ and $N$ as in Lemma \ref{thm:bound_zsample}.

For the second term in the bound on the Hellinger distance, the linearity of expectation, Tonelli's Theorem and Jensen's inequality give
%\begin{equation*}
%\frac{1}{2 \EE_{\nu_N^\mathcal G}\big(Z_{N,\mathcal G}^\mathrm{sample}\big)} II =  \left(Z^{-1/2} - \EE_{\nu_N^\mathcal G} \big(Z_{N,\mathcal G}^\mathrm{sample}\big)^{-1/2} \right)^2  \leq \max(Z^{-3},(\EE_{\nu_N^\mathcal G} \big(Z_{N,\mathcal G}^\mathrm{sample}\big))^{-3}) |Z - \EE_{\nu_N^\mathcal G} \big(Z_{N,\mathcal G}^\mathrm{sample}\big)|^2.
%\end{equation*}
%Using the linearity of expectation, Tonelli's Theorem and Jensen's inequality, we have
\begin{align*}
\left|Z - \EE_{\nu_N^\Phi} \big(Z_{N,\Phi}^\mathrm{sample}\big)\right|^2 \leq \int_X \left( \EE_{\nu_N^\Phi} \Big( \exp\big(-\Phi(u)\big) - \exp\big(-\Phi_N(u)\big)\Big) \right)^2 \mu_0(\mathrm{d}u),
\end{align*}
which can now be bounded as before. The second claim of the theorem then follows by {Lemma} \ref{thm:bound_zsample}.
\end{proof}

Similar to Theorem \ref{thm:hell_mean}, Theorem \ref{thm:hell_marginal} provides error bounds for general Gaussian process emulators of $\mathcal G$ and $\Phi$. An example of a Gaussian process emulator that satisfies the assumptions of Theorem \ref{thm:hell_marginal} is the emulator defined by \eqref{eq:pred_eq}, however, other choices are possible. {As in Corollary \ref{cor:rate_mean}, we can now combine Assumption \ref{ass:reg}, Theorem \ref{thm:hell_marginal} and Proposition \ref{prop:mean_conv} with $\beta = 0$ to derive error bounds in terms of the fill distance.}

\begin{corollary}\label{cor:rate_marginal} Suppose $\mathcal G_N$ and $\Phi_N$ are defined as in \eqref{eq:pred_eq}, with Mat\`ern kernel $k=k_{\nu,\lambda,\sigma_k^2}$. Suppose { Assumption A holds}, Assumption \ref{ass:reg} { holds with $s=\nu + K/2$,} and the assumptions of Proposition \ref{prop:mean_conv} and Theorem \ref{thm:hell_marginal} are satisfied. Then there exist constants $C_1, C_2, C_3$ and $C_4$, independent of $U$ and $N$, such that
\begin{align*}
\dhh(\mu^y, \mu^{y,N,\mathcal G}_\mathrm{marginal}) \leq C_1 h_U^{\nu + K/2} + C_2 h_U^\nu, \qquad
\text{and} \quad \dhh(\mu^y, \mu^{y,N, \Phi}_\mathrm{marginal}) \leq C_3 h_U^{\nu + K/2} + C_4 h_U^\nu.
\end{align*}
\end{corollary}
\begin{proof}We give the proof for $\mu^{y,N,\mathcal G}_\mathrm{marginal}$, the proof for $\mu^{y,N,\Phi}_\mathrm{marginal}$ is similar. Using Theorem \ref{thm:hell_marginal}, Jensen's inequality and the triangle inequality, we have
\begin{align*}
{\dhh(\mu^y, \mu^{y,N,\mathcal G}_\mathrm{marginal})^{2}} &\leq C \left\|\Big(\EE_{\nu_N^\mathcal G} \Big(\|\mathcal G  - \mathcal G_N \|^{2} \Big) \Big)^{1/2}\right\|^2_{L^2_{\mu_0}(X)} \\
&= C \int_X \EE_{\nu_N^\mathcal G} \Big(\|\mathcal G(u)  - \mathcal G_N(u) \|^{2} \Big) \mu_0(\mathrm{d}u) \\
&\leq 2C \int_X \|\mathcal G(u)  - m_N^\mathcal G(u) \|^{2} \mu_0(\mathrm{d}u) + 2C \int_X \EE_{\nu_N^\mathcal G} \Big(\|m_N^\mathcal G(u)  - \mathcal G_N(u) \|^{2} \Big) \mu_0(\mathrm{d}u).
\end{align*}
The first term can be bounded by using {Assumption A}, Assumption \ref{ass:reg}, Proposition \ref{prop:native_matern} and Proposition \ref{prop:mean_conv},
\begin{align*}
\int_X \|\mathcal G(u)  - m_N^\mathcal G(u) \|^{2} \mu_0(\mathrm{d}u) = \int_X \sum_{j=1}^J (\mathcal G^j(u) - m_N^{\mathcal G^j}(u))^2 \mu_0(\mathrm{d}u) 
\leq C h_U^{2\nu + K} \sum_{j=1}^J \| \mathcal G^j\|^2_{H^{\nu + K/2}(X)},
\end{align*}
for a constant $C$ independent of $U$ and $N$.
The second term can be bounded by using Assumption \ref{ass:reg}, Proposition \ref{prop:native_matern}, Proposition \ref{prop:mean_conv}, Proposition \ref{prop:predvar_sup}, the linearity of expectation and the Sobolev Embedding Theorem 
\begin{align*}
\int_X \EE_{\nu_N^\mathcal G} \Big(\|m_N^\mathcal G(u)  - \mathcal G_N(u) \|^{2} \Big) \mu_0(\mathrm{d}u) &= \int_X \EE_{\nu_N^\mathcal G} \Big(\sum_{j=1}^J (m_N^{\mathcal G^j}(u) - \mathcal G_N^j(u))^2 \Big) \mu_0(\mathrm{d}u) \\
&= J \int_X k_N(u,u) \mu_0(\mathrm{d}u) \\
&\leq J \sup_{u \in X} \sup_{\|g\|_{H_k}=1} | g(u) - m^g_N(u)|^2 \\
&\leq C h_U^{2\nu},
\end{align*}
for a constant $C$ independent of $U$ and $N$. The claim of the corollary then follows.
\end{proof}

{If Assumption \ref{ass:reg} holds only for some $s < \nu + K/2$, an analogue of Corollary \ref{cor:rate_marginal} can be proved using Proposition \ref{prop:mean_conv_int} with $\beta = 0$.

Note that the term $h_U^\nu$ appearing in the bounds in Corollary \ref{cor:rate_marginal} corresponds to the error bound on $\|k_N^{1/2}\|_{L^2(X)}$, which does not appear in the error bounds for $\mu^{y,N,\mathcal G}_\mathrm{mean}$ and $\mu^{y,N,\Phi}_\mathrm{marginal}$ analysed in Corollary \ref{cor:rate_mean}. Due to the supremum over $g$ appearing in the expression for $k_N(u,u)$ in Proposition \ref{prop:predvar_sup}, we can only conclude on the lower rate of convergence $h_U^\nu$ for $\|k_N^{1/2}\|_{L^2(X)}$. This result appears to be sharp, and the lower rate of convergence $\nu$ is observed in some of the numerical experiments in section \ref{sec:num} (cf Figures \ref{fig:marg} and \ref{fig:rand}).}

\subsubsection{Error in the random approximations $\mu^{y,N,\mathcal G}_\mathrm{sample}$ and $\mu^{y,N,\Phi}_\mathrm{sample}$}\label{ssec:gp_app_rand}

We have the following result for the random approximations $\mu^{y,N,\mathcal G}_\mathrm{sample}$ and $\mu^{y,N,\Phi}_\mathrm{sample}$.

\begin{theorem}\label{thm:hell_sample} Under the Assumptions of {Lemma} \ref{thm:bound_zsample}, there exist constants $C_1$ and $C_2$, independent of $U$ and $N$, such that for any $\delta > 0$,
\begin{align*}
\left(\EE_{\nu_N^\mathcal G} \left(\dhh(\mu^y, \mu^{y,N,\mathcal G}_\mathrm{sample})^2\right) \right)^{1/2} &\leq C_1 \left\|\Big( \EE_{\nu_N^\mathcal G} \Big(\|\mathcal G  - \mathcal G_N \|^{2+ \delta} \Big) \Big)^{1/(2+\delta)}\right\|_{L^2_{\mu_0}(X)}, \\
\left(\EE_{\nu_N^\Phi} \left(\dhh(\mu^y, \mu^{y,N,\Phi}_\mathrm{sample})^2\right) \right)^{1/2} &\leq C_2 \left\|\Big( \EE_{\nu_N^\Phi} \Big(|\Phi - \Phi_N |^{2+\delta} \Big) \Big)^{1/(2+\delta)}\right\|_{L^2_{\mu_0}(X)}.
\end{align*}
\end{theorem}
\begin{proof}
We start with $\mu^{y,N,\mathcal G}_\mathrm{sample}$. By the definition of the Hellinger distance and the linearity of expectation, we have
\begin{align*}
&\EE_{\nu_N^\mathcal G} \left(2  \; \dhh(\mu^y, \mu^{y,N,\mathcal G}_\mathrm{sample}) \right) ^2
= \EE_{\nu_N^\mathcal G} \left( \int_X \left( \sqrt{\frac{d\mu^y}{d\mu_0}} - \sqrt{\frac{d\mu^{y,N,\mathcal G}_\mathrm{sample}}{d\mu_0}} \right)^2 \mu_0(\mathrm{d}u) \right)\\
&\leq \frac{2}{Z} \EE_{\nu_N^\mathcal G} \left( \int_X \left(\exp \big(-\frac{1}{4 \sigma_\eta^2} \left\| y -  \mathcal G (u) \right\|^2 \big) - \exp \big(-\frac{1}{4 \sigma_\eta^2} \left\| y -  \mathcal G_N (u) \right\|^2 \big)\right)^2 \mu_0(\mathrm{d}u) \right) \\
& \qquad + 2 \; \EE_{\nu_N^\mathcal G} \left( Z_{N,\mathcal G}^\mathrm{sample} |Z^{-1/2} - (Z_{N, \mathcal G}^\mathrm{sample})^{-1/2} |^2 \right) \\
&=: I + II.
\end{align*}

For the first term, Tonelli's Theorem, the local Lipschitz continuity of the exponential function, the equality $a^2-b^2 = (a-b)(a+b)$, the reverse triangle inequality and H\"older's inequality with conjugate exponents $p=(1+\delta)/ \delta$ and $q = 1+ \delta$ give
\begin{align*}
&\frac{Z}{2} I =  \int_X  \EE_{\nu_N^\mathcal G} \left( \left( \exp \big(-\frac{1}{4 \sigma_\eta^2} \left\| y -  \mathcal G (u) \right\|^2 \big) - \exp \big(-\frac{1}{4 \sigma_\eta^2} \left\| y -  \mathcal G_N (u) \right\|^2 \big)\right)^2 \right) \mu_0(\mathrm{d}u)  \\
&\leq \frac{1}{\sigma_\eta^2}  \int_X  \EE_{\nu_N^\mathcal G} \left( \left( \| y -  \mathcal G(u)\|^2 - \| y -  \mathcal G_N(u) \|^2  \right)^2 \right) \mu_0(\mathrm{d}u)  \\
&\leq \frac{1}{\sigma_\eta^2} \int_X  \EE_{\nu_N^\mathcal G} \Big( \left(\left\| y -  \mathcal G_N (u) \right\| + \left\| y -  \mathcal G (u) \right\|\right)^2 \|\mathcal G_N (u) - \mathcal G (u)\|^2 \Big) \mu_0(\mathrm{d}u)  \\
&\leq \frac{1}{\sigma_\eta^2} \int_X  \Big(\EE_{\nu_N^\mathcal G} \left( \left( \| y -  \mathcal G(u)\| + \| y -  \mathcal G_N(u) \|\right)^{2(1+\delta) / \delta})\right) \Big)^{\delta/(\delta+1)} \Big(\EE_{\nu_N^\mathcal G} \left( \|\mathcal G(u) -  \mathcal G_N(u)\|^{2(1+\delta)}  \right) \Big)^{1/(1+\delta)} \mu_0(\mathrm{d}u) \\
&\leq \frac{1}{\sigma_\eta^2}  \sup_{u \in X}  \Big(\EE_{\nu_N^\mathcal G} \left( \left( \| y -  \mathcal G(u)\| + \| y -  \mathcal G_N(u) \|\right)^{2(1+\delta) / \delta})\right) \Big)^{\delta/(\delta+1)} \int_X \Big(\EE_{\nu_N^\mathcal G} \left( \|\mathcal G(u) -  \mathcal G_N(u)\|^{2(1+\delta)}  \right) \Big)^{1/(1+\delta)} \mu_0(\mathrm{d}u).
\end{align*}
for any $\delta > 0$. The supremum in the above bound can be bounded independently of $U$ and $N$ by Fernique's Theorem as in the proof of {Lemma} \ref{thm:bound_zsample}. It follows that there exists a constant $C$ independent of $U$ and $N$ such that
\[
\frac{Z}{2} I \leq C \; \left\|\Big(\EE_{\nu_N^\mathcal G} \Big(\|\mathcal G_N  - \mathcal G\|^{2(1+ \delta)} \Big) \Big)^{1/2(1+\delta)}\right\|^2_{L^2_{\mu_0}(X)}.
\]

For the second term in the bound on the Hellinger distance, we have
\begin{align*}
\frac{1}{2} II &= \EE_{\nu_N^\mathcal G} \left( Z_{N,\mathcal G}^\mathrm{sample} |Z^{-1/2} - (Z_{N, \mathcal G}^\mathrm{sample})^{-1/2} |^2 \right) \\
&\leq \EE_{\nu_N^\mathcal G} \left( Z_{N,\mathcal G}^\mathrm{sample} \max(Z^{-3},(Z_{N,\mathcal G}^\mathrm{sample})^{-3}) |Z - Z_{N,\mathcal G}^\mathrm{sample}|^2 \right).
\end{align*}
By Jensen's inequality and the same argument as above, we have
\begin{align*}
|Z - Z_{N,\mathcal G}^\mathrm{sample}|^2 &= \left| \int_X \left(\exp \big(-\frac{1}{4 \sigma_\eta^2} \left\| y -  \mathcal G (u) \right\|^2 \big) - \exp \big(-\frac{1}{4 \sigma_\eta^2} \left\| y -  \mathcal G_N (u) \right\|^2 \big)\right) \mu_0(\mathrm{d}u) \right|^2 \\
&\leq \frac{1}{\sigma_\eta^4} \int_X  \left(\left\| y -  \mathcal G_N (u) \right\| + \left\| y -  \mathcal G (u) \right\|\right)^2 \|\mathcal G_N (u) - \mathcal G (u)\|^2 \mu_0(\mathrm{d}u).
\end{align*}
Together with Tonelli's Theorem and H\"older's inequality with conjugate exponents $p=(1+\delta)/ \delta$ and $q = 1+ \delta$, we then have
\begin{align*}
&\frac{1}{2} II \\
&\leq \frac{1}{\sigma_\eta^4} \EE_{\nu_N^\mathcal G} \left( Z_{N,\mathcal G}^\mathrm{sample} \max(Z^{-3},(Z_{N,\mathcal G}^\mathrm{sample})^{-3}) \int_X  \left(\left\| y -  \mathcal G_N (u) \right\| + \left\| y -  \mathcal G (u) \right\|\right)^2 \|\mathcal G_N (u) - \mathcal G (u)\|^2 \mu_0(\mathrm{d}u) \right) \\
&= \frac{1}{\sigma_\eta^4} \int_X \EE_{\nu_N^\mathcal G} \left( Z_{N,\mathcal G}^\mathrm{sample} \max(Z^{-3},(Z_{N,\mathcal G}^\mathrm{sample})^{-3})   ( \left\| y -  \mathcal G_N (u) \right\| + \left\| y -  \mathcal G (u) \right\|)^2 \|\mathcal G_N (u) - \mathcal G (u)\|^2 \right) \mu_0(\mathrm{d}u) \\
&\leq \frac{1}{\sigma_\eta^4} \sup_{u \in X} \Big( \EE_{\nu_N^\mathcal G} \left( (Z_{N,\mathcal G}^\mathrm{sample})^{(1+\delta)/ \delta} \max(Z^{-3},(Z_{N,\mathcal G}^\mathrm{sample})^{-3})^{(1+\delta)/ \delta}  \left( \| y -  \mathcal G(u)\| + \| y -  \mathcal G_N(u) \|\right)^{2(1+\delta) / \delta})\right) \Big)^{\delta/(\delta+1)} \\
& \qquad \int_X \Big( \EE_{\nu_N^\mathcal G} \left( \|\mathcal G_N(u) -  \mathcal G(u)\|^{2(1+\delta)}  \right) \Big)^{1/(1+\delta)} \mu_0(\mathrm{d}u),
\end{align*}
for any $\delta > 0$. The supremum in the bound above can be bounded independently of $U$ and $N$ by {Lemma} \ref{thm:bound_zsample} and Fernique's Theorem. The first claim of the Theorem then follows.

The proof for $\mu^{y,N,\Phi}_\mathrm{sample}$ is similar. Using an identical corresponding splitting of the Hellinger distance $\EE_{\nu_N^\Phi} \left(2  \; \dhh^2(\mu^y, \mu^{y,N,\Phi}_\mathrm{sample}) \right) \leq I + II$, we bound the first term by Tonelli's Theorem, H\"older's inequality with conjugate exponents $p=(1+\delta)/ \delta$ and $q = 1+ \delta$ and the inequality $|\exp(x) - \exp(y)| \leq (\exp(x) + \exp(y)) |x-y|$, for all $x,y \in \R$,
\begin{align*}
&\frac{Z}{2} I =  \int_X  \EE_{\nu_N^\Phi} \left( \left( \exp \big(-\Phi(u)/2 \big) - \exp \big(-\Phi_N(u)/2 \big)\right)^2 \right) \mu_0(\mathrm{d}u) \\
&\leq \int_X \EE_{\nu_N^\Phi} \Big( |\exp\big(-\Phi(u)\big) + \exp\big(-\Phi_N(u)\big)|^2 |\Phi(u) - \Phi_N(u)|^2 \Big) \mu_0(\mathrm{d}u)\\
&\leq \int_X \EE_{\nu_N^\Phi} \Big( |\exp\big(-\Phi(u)\big) + \exp\big(-\Phi_N(u)\big)|^{2(1+\delta)/ \delta} \Big)^{\delta/(1+\delta)} \EE_{\nu_N^\Phi} \Big( |\Phi(u) - \Phi_N(u)|^{2(1+\delta)} \Big)^{1/(1+\delta)} \mu_0(\mathrm{d}u) \\
&\leq \sup_{u \in X} \EE_{\nu_N^\Phi} \Big( \big(\exp\big(-\Phi(u)\big) + \exp\big(-\Phi_N(u)\big)\big)^{2(1+\delta)/ \delta}\Big)^{\delta/(1+\delta)}   \left\|\EE_{\nu_N^\Phi} \left( |\Phi(u) - \Phi_N(u)|^{2(1+\delta)} \right)^{1/(2(1+\delta))}\right\|^2_{L^{2}_{\mu_0}(X)}.
\end{align*}
where the first factor can be bounded independently of $U$ and $N$ as in Lemma \ref{thm:bound_zsample}.

For the second term, we have as before
\begin{align*}
\frac{1}{2} II \leq \EE_{\nu_N^\Phi} \left( Z_{N,\Phi}^\mathrm{sample} \max(Z^{-3},(Z_{N,\Phi}^\mathrm{sample})^{-3}) |Z - Z_{N,\Phi}^\mathrm{sample}|^2 \right).
\end{align*}
%By Jensen's inequality and the same argument as above, we have
and
\begin{align*}
|Z - Z_{N,\Phi}^\mathrm{sample}|^2 &= \left| \int_X \left(\exp \big(-\Phi(u) \big) - \exp \big(-\Phi_N(u) \big)\right) \mu_0(\mathrm{d}u) \right|^2 \\
&\leq \int_X  \left(\exp \big(-\Phi(u) \big) + \exp \big(-\Phi_N(u) \big)\right)^2 (\Phi(u) - \Phi_N(u))^2 \mu_0(\mathrm{d}u).
\end{align*}
Together with Tonelli's Theorem and H\"older's inequality with conjugate exponents $p=(1+\delta)/ \delta$ and $q = 1+ \delta$, we then have
\begin{align*}
\frac{1}{2} II &\leq \EE_{\nu_N^\Phi} \left( Z_{N,\Phi}^\mathrm{sample} \max(Z^{-3},(Z_{N,\Phi}^\mathrm{sample})^{-3}) \int_X  \left(\exp \big(-\Phi(u) \big) + \exp \big(-\Phi_N(u) \big)\right)^2 (\Phi(u) - \Phi_N(u))^2 \mu_0(\mathrm{d}u) \right) \\
&= \int_X \EE_{\nu_N^\Phi} \left( Z_{N,\Phi}^\mathrm{sample} \max(Z^{-3},(Z_{N,\Phi}^\mathrm{sample})^{-3})\left(\exp \big(-\Phi(u) \big) + \exp \big(-\Phi_N(u) \big)\right)^2 (\Phi(u) - \Phi_N(u))^2 \right) \mu_0(\mathrm{d}u) \\
&\leq \left(\EE_{\nu_N^\mathcal G} \left( (Z_{N,\Phi}^\mathrm{sample})^{(1+\delta)/ \delta} \max(Z^{-3},(Z_{N,\Phi}^\mathrm{sample})^{-3})^{(1+\delta)/ \delta} \sup_{u \in X} \left(\exp \big(-\Phi(u) \big) + \exp \big(-\Phi_N(u) \big)\right)^{2(1+\delta)/ \delta} \right) \right)^{\delta/(\delta+1)} \\
& \qquad \int_X \Big(\EE_{\nu_N^\Phi} \left( \|\Phi(u) -  \Phi_N(u)\|^{2(1+\delta)}  \right) \Big)^{1/(1+\delta)} \mu_0(\mathrm{d}u),
\end{align*}
for any $\delta > 0$. The first expected value in the bound above can be bounded independently of $U$ and $N$ by {Lemma} \ref{thm:bound_zsample}. The second claim of the Theorem then follows.
\end{proof}

Similar to Theorem \ref{thm:hell_mean} and Theorem \ref{thm:hell_marginal}, Theorem \ref{thm:hell_sample} provides error bounds for general Gaussian process emulators of $\mathcal G$ and $\Phi$. As a particular example, we can take the emulators defined by \eqref{eq:pred_eq}. {We can now combine Assumption \ref{ass:reg}, Theorem \ref{thm:hell_sample} and Proposition \ref{prop:mean_conv} with $\beta = 0$ to derive error bounds in terms of the fill distance.}

\begin{corollary}\label{cor:rate_sample} Suppose $\mathcal G_N$ and $\Phi_N$ are defined as in \eqref{eq:pred_eq}, with Mat\`ern kernel $k=k_{\nu,\lambda,\sigma_k^2}$. Suppose { Assumption A holds}, Assumption \ref{ass:reg} { holds with $s=\nu + K/2$,} and the assumptions of Proposition \ref{prop:mean_conv} and Theorem \ref{thm:hell_sample} are satisfied. Then there exist constants $C_1, C_2, C_3$ and $C_4$, independent of $U$ and $N$, such that
\begin{align*}
\dhh(\mu^y, \mu^{y,N,\mathcal G}_\mathrm{marginal}) \leq C_1 h_U^{\nu + K/2} + C_2 h_U^\nu, \qquad
\text{and} \quad \dhh(\mu^y, \mu^{y,N, \Phi}_\mathrm{marginal}) \leq C_3 h_U^{\nu + K/2} + C_4 h_U^\nu.
\end{align*}
\end{corollary}
\begin{proof} The proof is similar to that of Corollary \ref{cor:rate_marginal}, exploiting that for a Gaussian random variable $X$, we have $\EE((X- \EE(X))^4) = 3 (\EE((X-\EE(X))^2))^2$.
\end{proof}

{If Assumption \ref{ass:reg} holds only for some $s < \nu + K/2$, an analogue of Corollary \ref{cor:rate_sample} can be proved using Proposition \ref{prop:mean_conv_int} with $\beta = 0$.}

We furthermore have the following result on a generalised total variation distance \cite{rh15}, defined by
\[
d_\mathrm{gTV}(\mu^y, \mu^{y,N,\mathcal G}_\mathrm{sample}) = \sup_{\|f\|_{C^0(X)}\leq 1} \Big( \EE_{\nu^\mathcal G_N} \big( |\EE_{\mu^y}(f) - \EE_{\mu^{y,N,\mathcal G}_\mathrm{sample}} (f)|^2 \big) \Big)^{1/2},
\]
for $\mu^{y,N,\mathcal G}_\mathrm{sample}$, and defined analogously for $\mu^{y,N,\Phi}_\mathrm{sample}$.

\begin{theorem}\label{thm:gtv_sample} Under the Assumptions of {Lemma} \ref{thm:bound_zsample}, there exist constants $C_1$ and $C_2$, independent of $U$ and $N$, such that for any $\delta > 0$ 
\begin{align*}
d_\mathrm{gTV}(\mu^y, \mu^{y,N,\mathcal G}_\mathrm{sample}) &\leq C_1 \; \left\| \Big(\EE_{\nu_N^\mathcal G} \Big(\|\mathcal G - \mathcal G_N\|^{2+ \delta} \Big) \Big)^{1/(2+\delta)}\right\|_{L^2_{\mu_0}(X)}, \\ 
d_\mathrm{gTV}(\mu^y, \mu^{y,N,\Phi}_\mathrm{sample}) &\leq C_2 \; \left\|\Big(\EE_{\nu_N^\mathcal G} \Big(\|\Phi - \Phi_N \|^{2+\delta} \Big)\Big)^{1/(2+\delta)}\right\|_{L^2_{\mu_0}(X)}.
\end{align*}
\end{theorem}
\begin{proof} We give the proof for $\mu^{y,N,\Phi}_\mathrm{sample}$; the proof for $\mu^{y,N,\mathcal G}_\mathrm{sample}$ is identical. By definition, we have
\begin{align*}
&d_\mathrm{gTV}(\mu^y, \mu^{y,N,\Phi}_\mathrm{sample}) = \sup_{\|f\|_{C^0(X)}\leq 1} \Big( \EE_{\nu^\Phi_N} \big( |\EE_{\mu^y}(f) - \EE_{\mu^{y,N,\Phi}_\mathrm{sample}} (f)|^2 \big) \Big)^{1/2} \\
&= \sup_{\|f\|_{C^0(X)}\leq 1} \left( \EE_{\nu^\Phi_N} \left( \left|\int_X f(u) \left(\exp(-\Phi(u)) Z^{-1} - \exp (-\Phi_N(u)) (Z_{N,\Phi}^\mathrm{sample})^{-1} \right) \mu_0(\mathrm{d}u)\right|^2 \right) \right)^{1/2} \\
&\leq \left( \EE_{\nu^\Phi_N} \left( \left|\int_X \left(\exp(-\Phi(u)) Z^{-1} - \exp (-\Phi_N(u)) (Z_{N,\Phi}^\mathrm{sample})^{-1} \right) \mu_0(\mathrm{d}u)\right|^2 \right) \right)^{1/2} \\
&\leq \frac{2}{Z} \left( \EE_{\nu^\Phi_N} \left( \int_X |\exp(-\Phi(u)) - \exp (-\Phi_N(u))|^{2} \mu_0(\mathrm{d}u) \right) \right)^{1/2} + \\
&\qquad \left( \EE_{\nu^\Phi_N} \left( (Z_{N,\Phi}^\mathrm{sample})^{2} | Z^{-1} - (Z_{N,\Phi}^\mathrm{sample})^{-1}|^2 \right) \right)^{1/2} \\
&=: I + II.
\end{align*}
The terms $I$ and $II$ can be bounded by the same arguments as the terms $I$ and $II$ in the proof of Theorem \ref{thm:hell_sample}, by noting that $| Z^{-1} - (Z_{N,\Phi}^\mathrm{sample})^{-1}|^2 \leq \max(Z^{-4},(Z_{N,\Phi}^\mathrm{sample})^{-4}) |Z - Z_{N,\Phi}^\mathrm{sample}|^2$. 
\end{proof}

\section{Numerical Examples}\label{sec:num}
{We consider the model inverse problem of determining the diffusion coefficient of an elliptic partial differential equation (PDE) in divergence form from observation of a finite set of noisy continuous functionals of the solution.} This type of equation arises, for example, in the modelling of groundwater flow in a porous medium. We consider the one-dimensional model problem
{\begin{equation}\label{eq:mod}
-\frac{\mathrm d}{\mathrm{d} x} \Big(\kappa(x;u) \frac{\mathrm d p}{\mathrm d x} (x;u)\Big) = 1 \quad \text{in } (0,1), \qquad  p(1;u) = p(0;u) = 0,
\end{equation}}
where the coefficient $\kappa$ depends on parameters $u = \{u_j\}_{j=1}^K \in [-1,1]^K$ through the linear expansion 
{\[
\kappa(x;u) = \frac{1}{100} + \sum_{j=1}^K \frac{u_j}{200 (K+1)} \sin(2 \pi j x).
\]}
{In this setting the forward map $G : [-1,1]^K \rightarrow H^1_0(D)$, defined by $G(u) = p$, is an analytic function \cite{cohen2011analytic}. Since the observation operator $\mathcal O$ is linear and bounded, Assumption \ref{ass:reg} is satisfied for any $s > K/2$. 

Unless stated otherwise, we will throughout this section approximate the solution $p$ by standard, piecewise linear, continuous finite elements on a uniform grid with mesh size $h=1/32$. The corresponding approximate forward map, denoted by $G_h$, is also an analytic function of $u$ \cite{cohen2011analytic}, and Assumption \ref{ass:reg} is satisfied for any $s > K/2$ also for $G_h$. By slight abuse of notation, we will denote the posterior measure corresponding to the forward map $G_h$ by $\mu^{y}$, and use this as our reference measure. The error induced by the finite element approximation will be ignored.}

As prior measure $\mu_0$ on $[-1,1]^K$, we use the uniform product measure $\mu_0(\mathrm{d}u) = \bigotimes_{j =1}^K \frac{\mathrm d u_j}{2}$. 
The observations $y$ are taken as noisy point evaluations of the solution, $y_j = p(x_j; u^*) + \eta_j$ with $\eta \sim \mathcal N(0,I)$ and $\{x_j\}_{j=1}^J$ evenly spaced points in $(0,1)$. To generate $y$, the truth $u^*$ was chosen as a random sample from the prior, and the solution $p$ was approximated by finite elements on a uniform grid with mesh size $h^*=1/1024$.

The emulators $\mathcal G_N$ and $\Phi_N$ are computed as described in section \ref{ssec:gp_sk}, with mean and covariance kernel given by \eqref{eq:pred_eq}. In the Gaussian process prior \eqref{eq:gp}, we choose $m \equiv 0$ and $k = k_{\nu,1,1}$, a Mat\`ern kernel with variance $\sigma_k^2=1$, correlation length $\lambda=1$ and smoothness parameter $\nu$. 

For a given approximation $\mu^{y,N}$ to $\mu^{y}$, we will compute twice the Hellinger distance squared,
{\begin{equation*}
2 \dhh(\mu^{y}, \mu^{y,N})^2 = \int_{[-1,1]^K} \left(\sqrt{\frac{d \mu^{y}}{d \mu_0}}(u) - \sqrt{\frac{d \mu^{y,N}}{d \mu_0}}(u) \right)^2 d \mu_0(u).
\end{equation*}}
The integral over $[-1,1]^K$ is approximated by a randomly shifted lattice rule with product weight parameters $\gamma_j=1/j^2$ \cite{niederreiter}. The generating vector for the rule used is available from Frances Kuo's website (\texttt{http://web.maths.unsw.edu.au/$\sim$fkuo/}) as ``lattice-39102-1024-1048576.3600''. For the marginal and random approximations, the expected value over the Gaussian process is approximated by Monte Carlo sampling, using the MATLAB command \texttt{mvnrnd}. 

For the design points $U$, we choose a uniform tensor grid. In $[-1,1]^K$, the uniform tensor grid consisting of $N = N_*^K$ points, for some $N_* \in \mathbb N$, has fill distance $h_U = \sqrt{K} (N_* - 1)^{-1}$. In Table \ref{tbl:conv}, we give the convergence rates in $N$ for $\sup_{u \in X} | f(u) - m_N^f(u)|^2$ and $\| f - m_N^f\|_{L^2(X)}^2$ predicted by Proposition \ref{prop:mean_conv}.

\begin{table} [p]
\begin{center}
\renewcommand{\arraystretch}{1.25}
\begin{tabular}{ |c|| c c c c||c|| c c c c|} \hline 
\multicolumn{5}{|c||}{$\sup_{u \in X} | f(u) - m_N^f(u)|^2$} &\multicolumn{5}{|c|}{$\| f - m_N^f\|_{L^2(X)}^2$} \\ \hline
\backslashbox{$\nu$}{$K$}& 1& 2& 3& 4 & \backslashbox{$\nu$}{$K$}& 1& 2& 3& 4\\ \hline 
1&  2& 1& 0.67&  0.5 & 1&  3& 2& 1.7&  1.5\\
5&  & 5& 3.3&  & 5&  & 6& 4.3& \\ \hline
\end{tabular}
\end{center}
\caption{Convergence rates in $N$ predicted by Proposition \ref{prop:mean_conv} for uniform tensor grids.}
\label{tbl:conv}
\end{table}

\begin{table} [p]
\begin{center}
\renewcommand{\arraystretch}{1.25}
\begin{tabular}{ |c|| c c ||c|| c c ||c|| c c|} \hline 
\multicolumn{3}{|c||}{$\mu^{y,N,\mathcal G}_\mathrm{mean}$} &\multicolumn{3}{|c|}{$\mu^{y,N,\mathcal G}_\mathrm{marginal}$} &\multicolumn{3}{|c|}{$\mu^{y,N,\mathcal G}_\mathrm{sample}$}\\ \hline
\backslashbox{$\nu$}{$K$}& 2& 3 & \backslashbox{$\nu$}{$K$}& 2& 3 & \backslashbox{$\nu$}{$K$}& 2& 3\\ \hline 
1&  2.6& 2.4& 1&  2.6 & 2.2&  1& 2.3& 1.7 \\
5&  6.2& 4.5& 5 & 6.2& 4.6 & 5& 6.1& 4.4\\ \hline
\end{tabular}
\end{center}
\caption{Observed convergence rates in $N$ of $\dhh(\mu^{y}, \mu^{y,N,\mathcal G})^2$, as shown in Figures \ref{fig:mean}, \ref{fig:marg} and \ref{fig:rand}.}
\label{tbl:obs_conv_G_1}
\end{table}

\begin{table} [p]
\begin{center}
\renewcommand{\arraystretch}{1.25}
\begin{tabular}{ |c|| c c ||c|| c c ||c|| c c|} \hline 
\multicolumn{3}{|c||}{$\mu^{y,N,\Phi}_\mathrm{mean}$} &\multicolumn{3}{|c|}{$\mu^{y,N,\Phi}_\mathrm{marginal}$} &\multicolumn{3}{|c|}{$\mu^{y,N,\Phi}_\mathrm{sample}$}\\ \hline
\backslashbox{$\nu$}{$K$}& 2& 3 & \backslashbox{$\nu$}{$K$}& 2& 3 & \backslashbox{$\nu$}{$K$}& 2& 3\\ \hline 
1&  2.5& 2& 1&  1.8 & 1.1&  1& 1.1& 0.76 \\
5&  5.4& 3.8& 5 & 4.9& 3.2 & 5& 4.9& 3.3\\ \hline
\end{tabular}
\end{center}
\caption{Observed convergence rates in $N$ of $\dhh(\mu^{y}, \mu^{y,N,\Phi})^2$, as shown in Figures \ref{fig:mean}, \ref{fig:marg} and \ref{fig:rand}.}
\label{tbl:obs_conv_Phi_1}
\end{table}

\begin{table} [p]
\begin{center}
\renewcommand{\arraystretch}{1.25}
\begin{tabular}{ |c|| c c c c||c|| c c c c|} \hline 
\multicolumn{5}{|c||}{$\mu^{y,N,\mathcal G}_\mathrm{mean}$} &\multicolumn{5}{|c|}{$\mu^{y,N,\Phi}_\mathrm{mean}$} \\ \hline
\backslashbox{$\nu$}{$K$}&1 & 2& 3 & 4& \backslashbox{$\nu$}{$K$}&1 & 2& 3 & 4\\ \hline 
1&  4.1& 2.7& 2.3&  2.3 & 1& 4&  2.7& 2.1& 1.9 \\ \hline
\end{tabular}
\end{center}
\caption{Observed convergence rates in $N$ of $\dhh(\mu^{y}, \mu^{y,N,\Phi})^2$ and $\dhh(\mu^{y}, \mu^{y,N,\mathcal G})^2$, as shown in Figure \ref{fig:mean_15}.}
\label{tbl:obs_conv_15}
\end{table}

\subsection{Mean-based approximations}
In Figure \ref{fig:mean}, we show $2 \dhh(\mu^y, \mu^{y,N,\mathcal G}_\mathrm{mean})^2$ (left) and $2 \dhh(\mu^y, \mu^{y,N,\Phi}_\mathrm{mean})^2$ (right), for a variety of choices of $K$ and $\nu$, for $J=1$. For each choice of the parameters $K$ and $\nu$, we have as a dotted line added the least squares fit of the form $C_1 N^{-C_2}$, for some $C_1, C_2 > 0$, and indicated the rate $N^{-C_2}$ in the legend. { The observed rates $C_2$ are also summarised in Tables \ref{tbl:obs_conv_G_1} and \ref{tbl:obs_conv_Phi_1}.} By Corollary \ref{cor:rate_mean}, we expect to see the faster convergence rates in the right panel of Table \ref{tbl:conv}. {For convenience, we have added these rates in parentheses in the legends in Figure \ref{fig:mean}}. For $\mu^{y,N,\mathcal G}_\mathrm{mean}$, we observe the rates in Table \ref{tbl:conv}, or slightly faster. For $\mu^{y,N,\Phi}_\mathrm{mean}$, we observe rates slightly faster than predicted for $\nu=1$, and slightly slower than predicted for $\nu=5$.  Finally, we remark that though the convergence rates of the error are slightly slower for $\mu^{y,N,\Phi}_\mathrm{mean}$, the actual errors are smaller for $\mu^{y,N,\Phi}_\mathrm{mean}$.

In Figure \ref{fig:mean_15}, we again show $2 \dhh(\mu^y, \mu^{y,N,\mathcal G}_\mathrm{mean})^2$ (left) and $2 \dhh(\mu^y, \mu^{y,N,\Phi}_\mathrm{mean})^2$ (right), for a variety of choices of $K$, with $J=15$ and $\nu=1$. { The observed convergence rates are summarised in Table \ref{tbl:obs_conv_15}.} We again observe convergence rates slightly faster than the rates predicted in the right panel of Table \ref{tbl:conv}. As in Figure \ref{fig:mean}, we observe that the errors in $\mu^{y,N,\Phi}_\mathrm{mean}$ are smaller, though the rates of convergence are slightly faster for $\mu^{y,N,\mathcal G}_\mathrm{mean}$.

\subsection{Marginal approximations}
In Figure \ref{fig:marg}, we show $2 \dhh(\mu^y, \mu^{y,N,\mathcal G}_\mathrm{marginal})^2$ (left) and $2 \dhh(\mu^y, \mu^{y,N,\Phi}_\mathrm{marginal})^2$ (right), for a variety of choices of $K$ and $\nu$, for $J=1$. For each choice of the parameters $K$ and $\nu$, we have again added the least squares fit of the form $C_1 N^{-C_2}$, and indicated the rate $C_2$ in the legend. { The observed rates $C_2$ are also summarised in Tables \ref{tbl:obs_conv_G_1} and \ref{tbl:obs_conv_Phi_1}.} By Corollary \ref{cor:rate_marginal}, we expect the error to be the sum of two contributions, one of which decays at the rate indicated in the left panel of Table \ref{tbl:conv}, and another which decays at the rate indicated by the right  panel of Table \ref{tbl:conv}. {For convenience, we have added these rates in parentheses in the legends in Figure \ref{fig:marg}}.For $\mu^{y,N,\mathcal G}_\mathrm{marginal}$, we observe the faster convergence rates in the right panel of Table \ref{tbl:conv}, although a closer inspection indicates that the convergence is slowing down as $N$ increases. For $\mu^{y,N,\mathcal G}_\mathrm{marginal}$, the observed rates are somewhere between the two rates predicted by Table \ref{tbl:conv}.

\subsection{Random approximations}
In Figure \ref{fig:rand}, we show $2 \EE_{\nu_N^\mathcal G} (\dhh(\mu^y, \mu^{y,N,\mathcal G}_\mathrm{sample})^2)$ (left) and $2 \EE_{\nu_N^\Phi} (\dhh(\mu^y, \mu^{y,N,\Phi}_\mathrm{sample})^2)$ (right), for a variety of choices of $K$ and $\nu$, for $J=1$. For each choice of the parameters $K$ and $\nu$, we have again added the least squares fit of the form $C_1 N^{-C_2}$, and indicated the rate $C_2$ in the legend. { The observed rates $C_2$ are also summarised in Tables \ref{tbl:obs_conv_G_1} and \ref{tbl:obs_conv_Phi_1}.} By Corollary \ref{cor:rate_sample}, we expect the error to be the sum of two contributions, as for the marginal approximations considered in the previous section, {and the corresponding rates from Table \ref{tbl:conv} have been added in parentheses in the legends}. For $\mu^{y,N,\mathcal G}_\mathrm{sample}$, we again observe the faster convergence rates in the right panel of Table \ref{tbl:conv}, but the convergence again seems to be slowing down as $N$ increases. For $\mu^{y,N,\mathcal G}_\mathrm{marginal}$, the observed rates are very close to the slower rates in the left panel of Table \ref{tbl:conv}.

\begin{figure}[p]
\centering
\hspace*{-0.75cm}\includegraphics[width=0.5\textwidth]{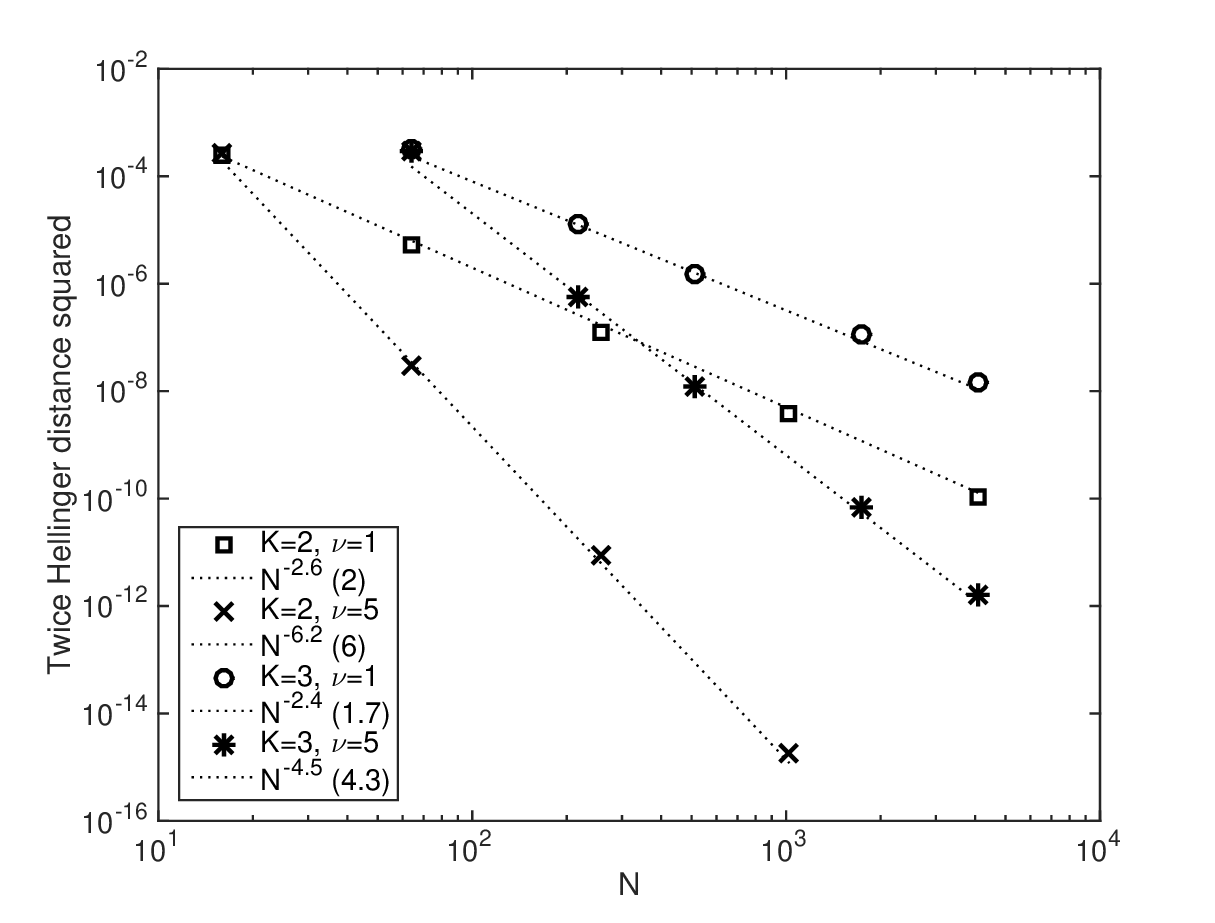}\ \includegraphics[width=0.5\textwidth]{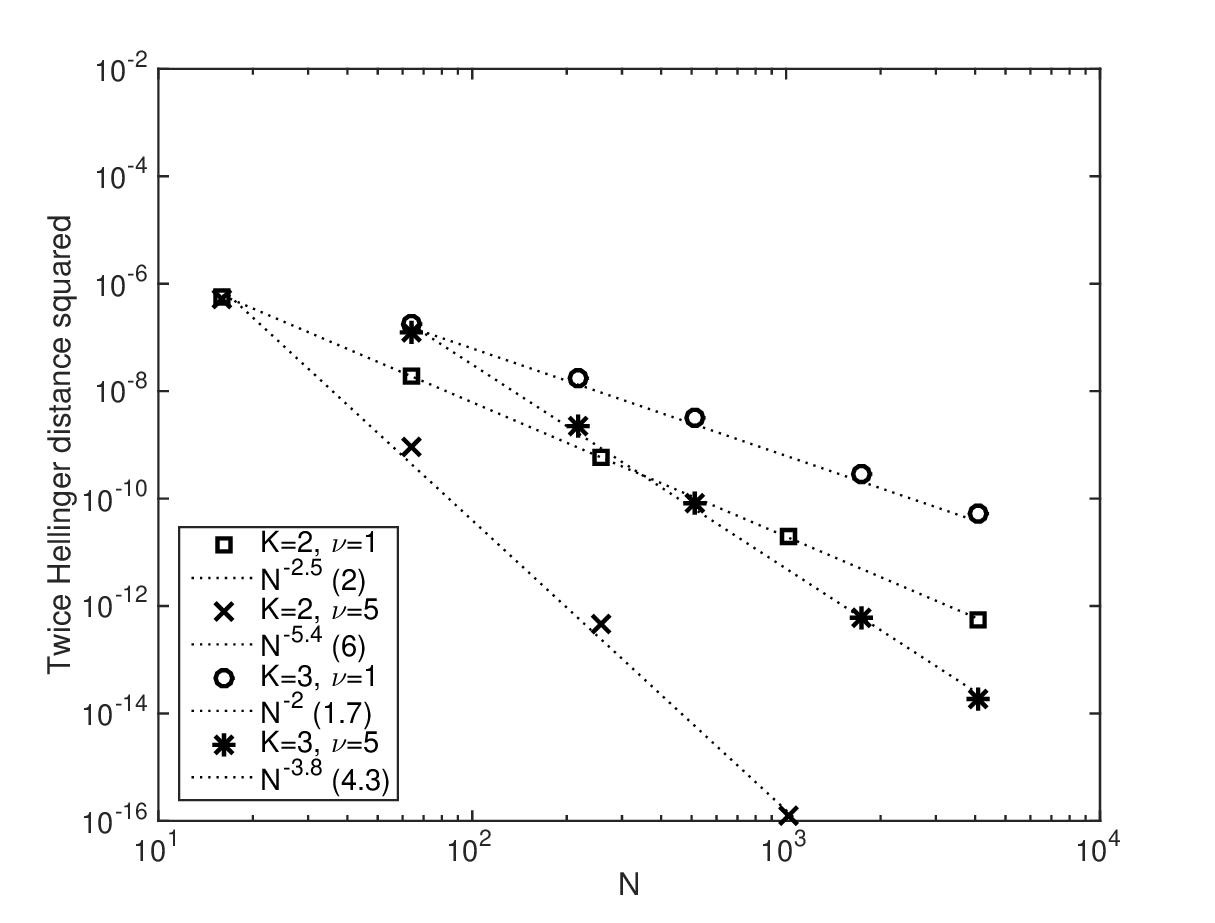}\hspace*{-0.75cm} 
\caption{$2 \dhh(\mu^y, \mu^{y,N,\mathcal G}_\mathrm{mean})^2$ (left) and $2 \dhh(\mu^y, \mu^{y,N,\Phi}_\mathrm{mean})^2$ (right), for a variety of choices of $K$ and $\nu$, for $J=1$.}
\label{fig:mean}
\end{figure} 

\begin{figure}[p]
\centering
\hspace*{-0.75cm}\includegraphics[width=0.5\textwidth]{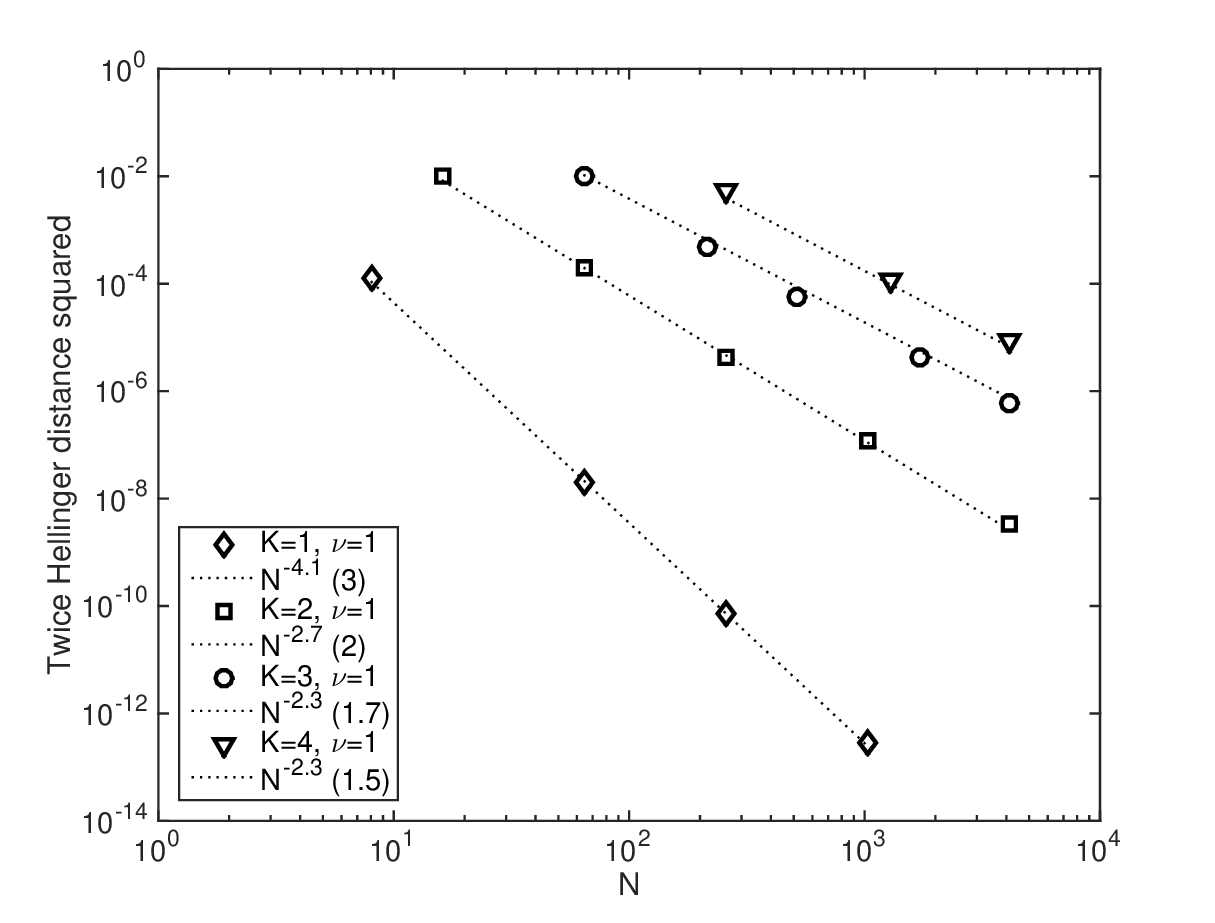}\ \includegraphics[width=0.5\textwidth]{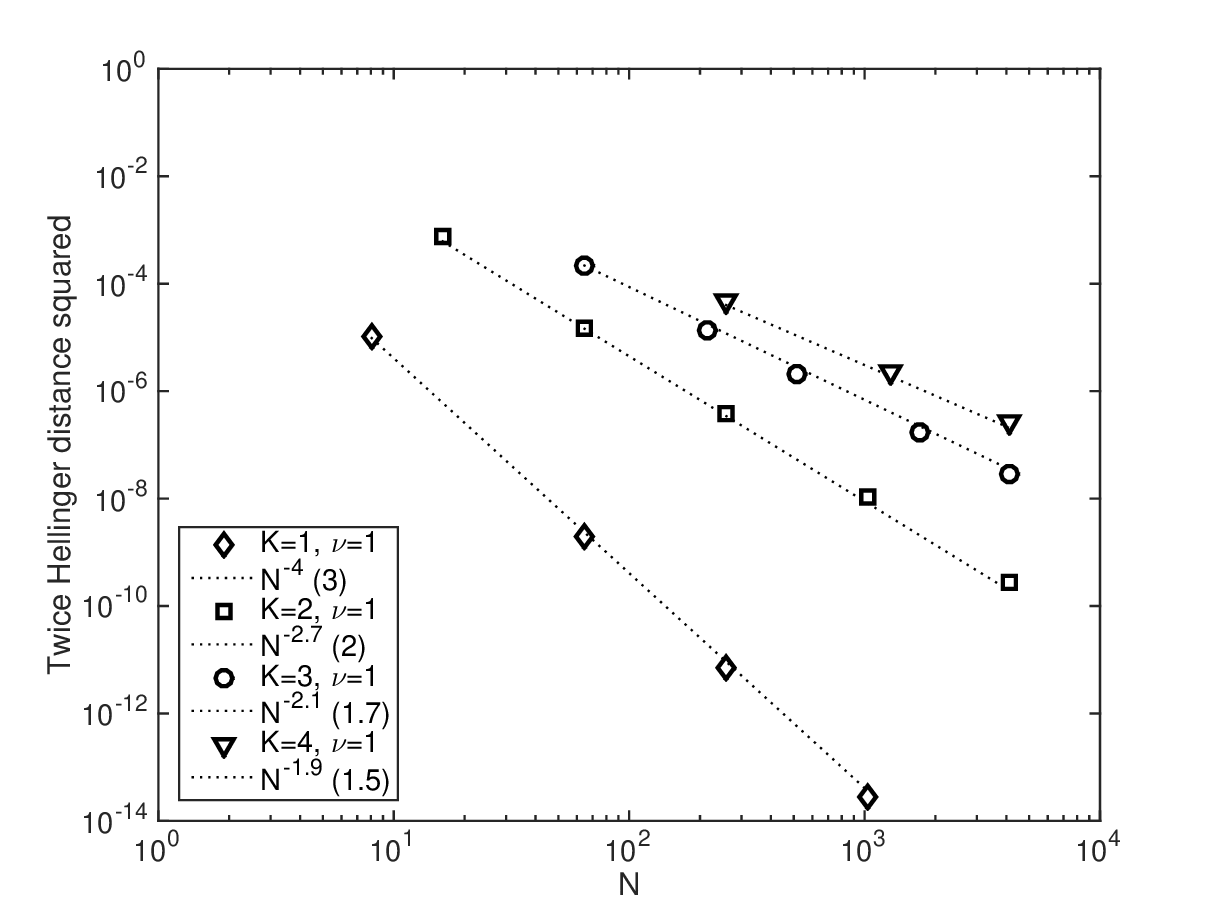}\hspace*{-0.75cm} 
\caption{$2 \dhh(\mu^y, \mu^{y,N,\mathcal G}_\mathrm{mean})^2$ (left) and $2 \dhh(\mu^y, \mu^{y,N,\Phi}_\mathrm{mean})^2$ (right), for a variety of choices of $K$ and $\nu=1$, for $J=15$.}
\label{fig:mean_15}
\end{figure} 

\begin{figure}[p]
\centering
\hspace*{-0.75cm}\includegraphics[width=0.5\textwidth]{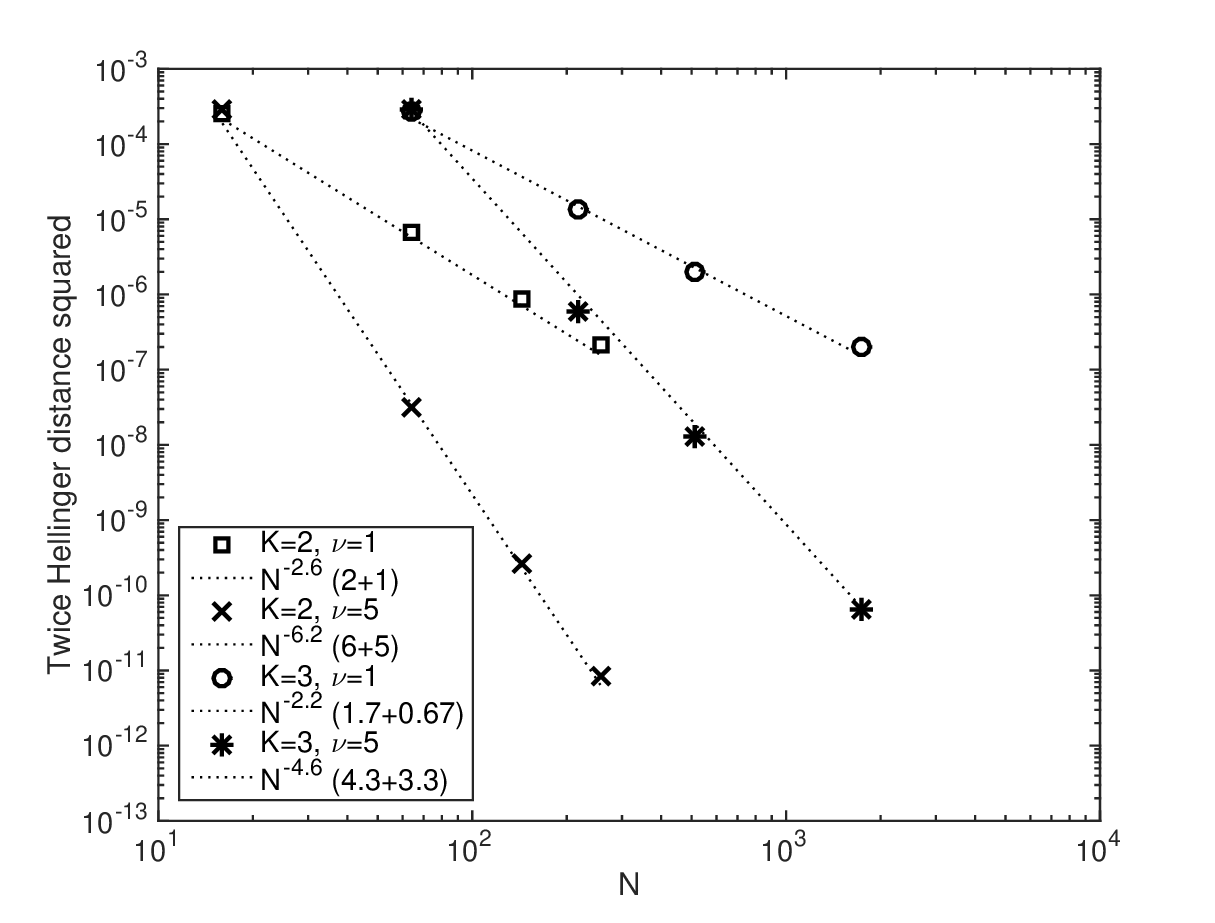}\ \includegraphics[width=0.5\textwidth]{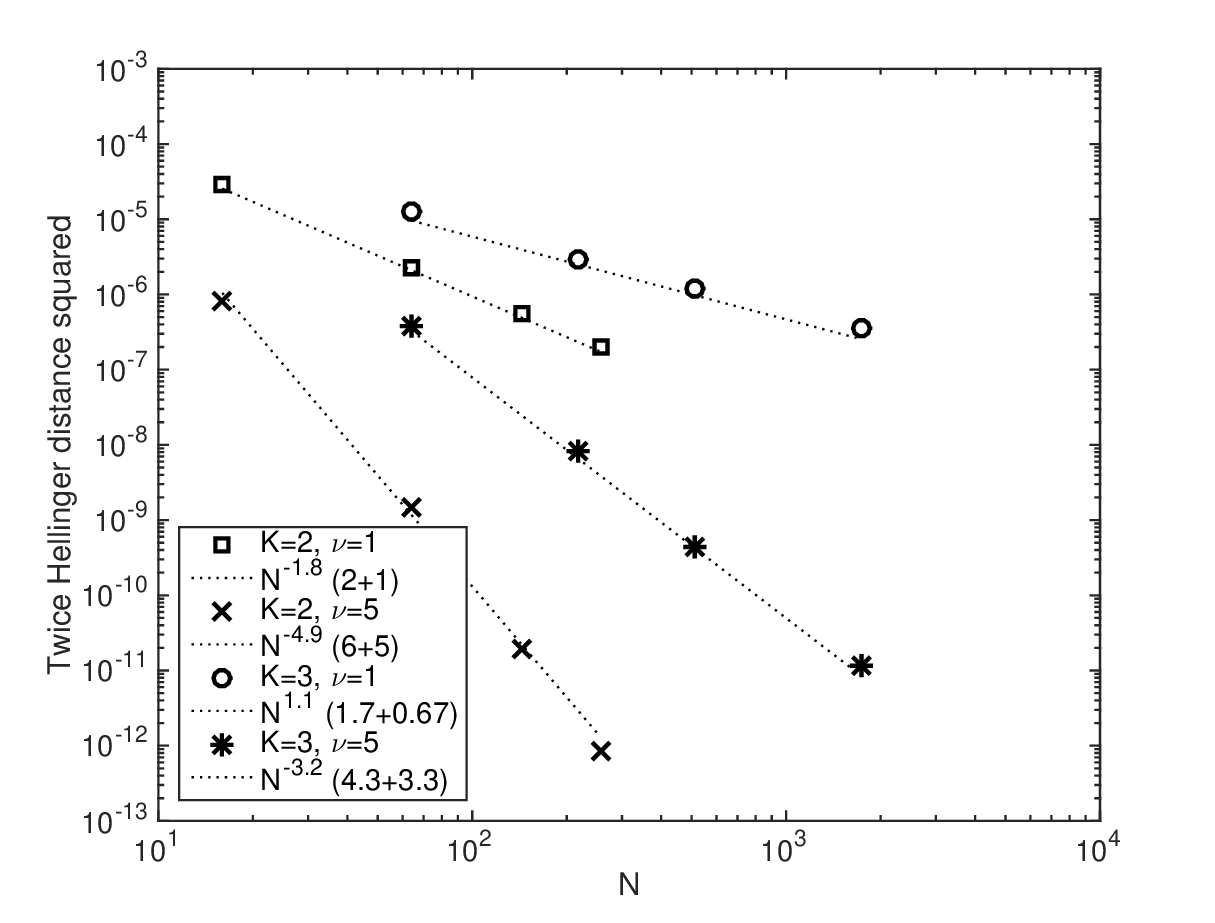}\hspace*{-0.75cm} 
\caption{$2 \dhh(\mu^y, \mu^{y,N,\mathcal G}_\mathrm{marginal})^2$ (left) and $2 \dhh(\mu^y, \mu^{y,N,\Phi}_\mathrm{marginal})^2$ (right), for a variety of choices of $K$ and $\nu$, for $J=1$.}
\label{fig:marg}
\end{figure} 

\begin{figure}[p]
\centering
\hspace*{-0.75cm}\includegraphics[width=0.5\textwidth]{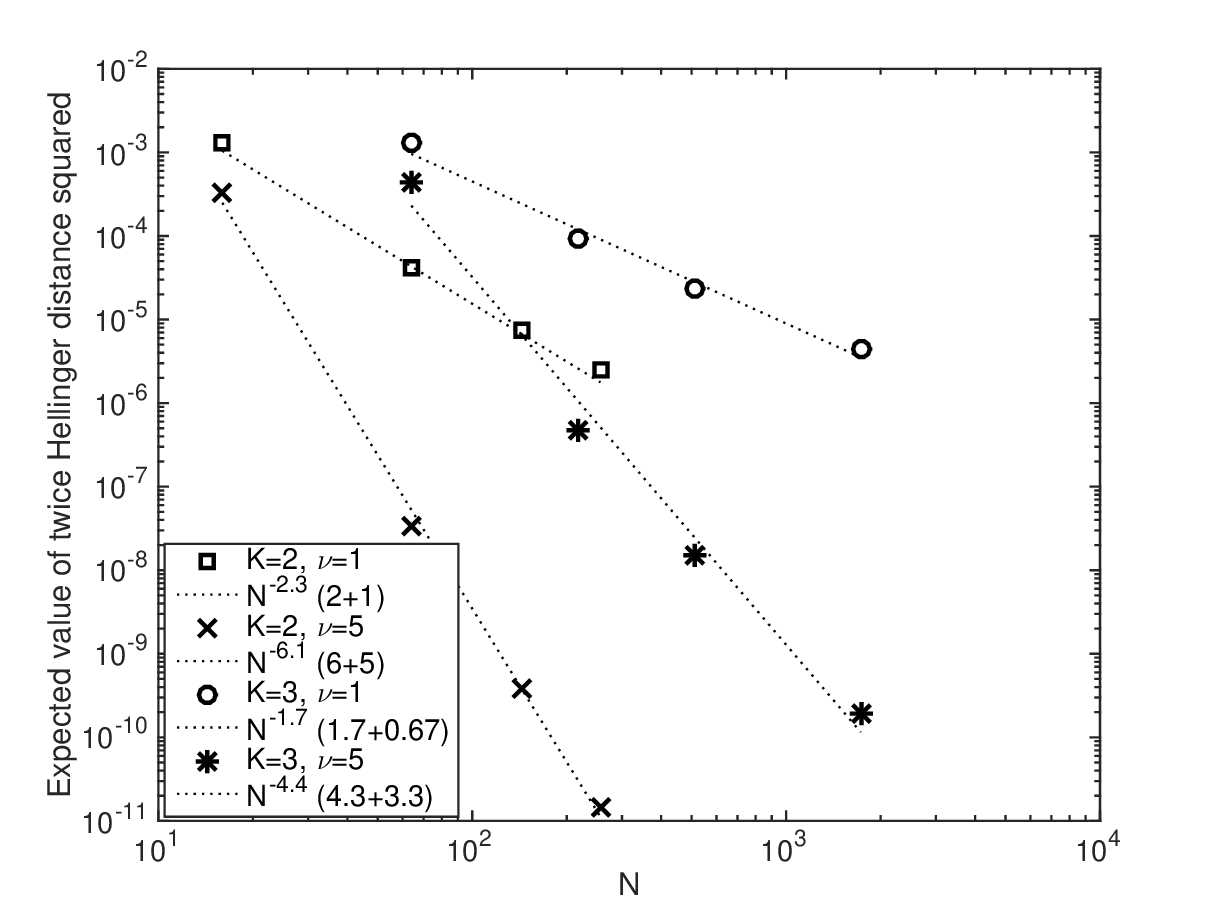}\ \includegraphics[width=0.5\textwidth]{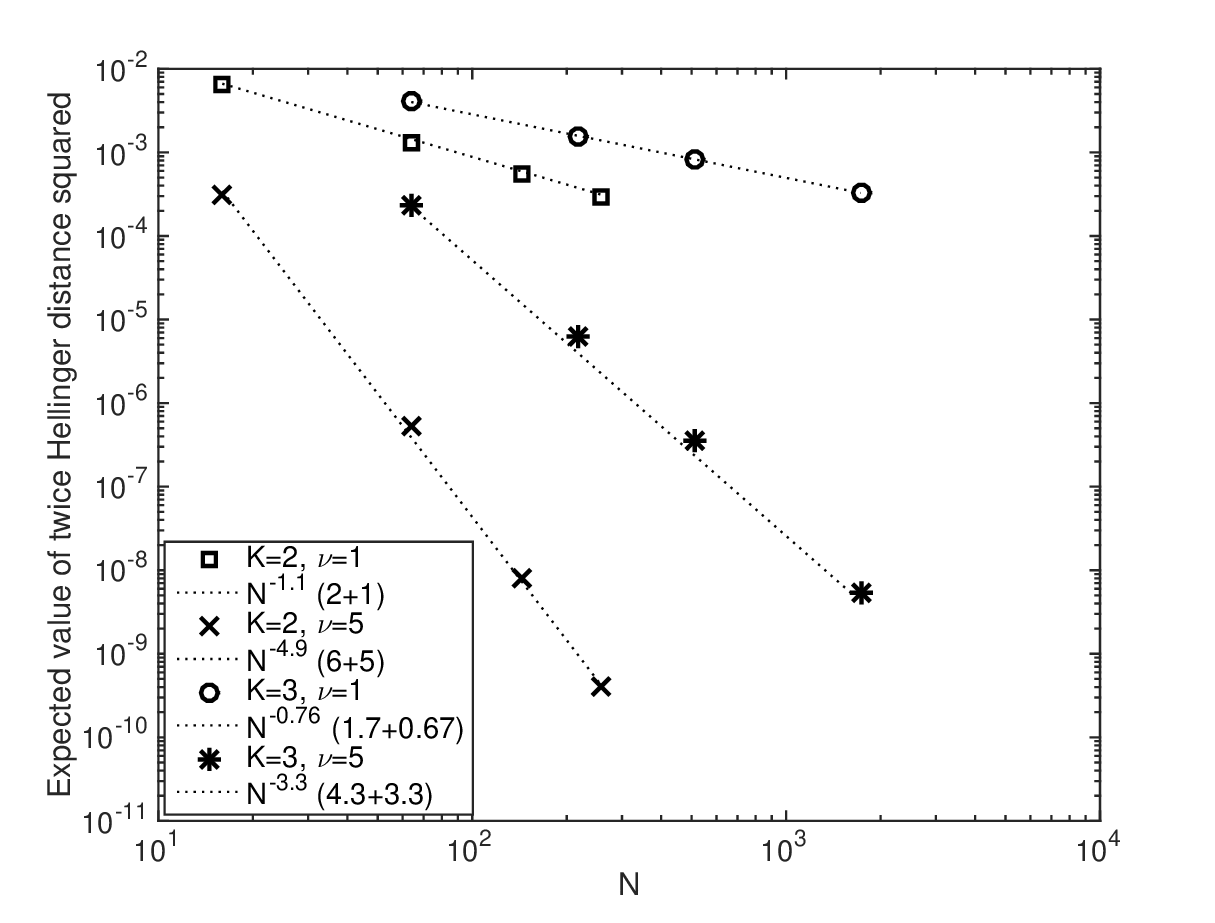}\hspace*{-0.75cm} 
\caption{$2 \EE_{\nu_N^\mathcal G} (\dhh(\mu^y, \mu^{y,N,\mathcal G}_\mathrm{sample})^2)$ (left) and $2 \EE_{\nu_N^\Phi} (\dhh(\mu^y, \mu^{y,N,\Phi}_\mathrm{sample})^2)$ (right), for a variety of choices of $K$ and $\nu$, for $J=1$.}
\label{fig:rand}
\end{figure}

\section{Conclusions and further work}\label{sec:conc}
Gaussian process emulators are frequently used as surrogate models. In this work, we analysed the error that is introduced in the Bayesian posterior distribution when a Gaussian process emulator is used to approximate the forward model, either in terms of the parameter-to-observation map or the negative log-likelihood. We showed that the error in the posterior distribution, measured in the Hellinger distance, can be bounded in terms of the error in the emulator, measured in a norm dependent on the approximation considered.

An issue that requires further consideration is the efficient emulation of vector-valued functions. A simple solution, employed in this work, is to emulate each entry independently. In many applications, {however, it is natural} to assume that the entries are correlated, and a better emulator could be constructed by including this correlation in the emulator. {Furthermore}, there are still a lot of open questions about how to do this optimally \cite{bzkl13}. {Also the question of scaling the Gaussian process methodology
to high dimensional input spaces remains open. The current error bounds from scattered data approximation employed in this paper feature a strong dependence on the input dimension $K$, yielding poor convergence estimates in high dimensions.}

{Another important issue is the selection of the design points used to construct the Gaussian process emulator, also known as experimental design. In applications where the posterior distribution concentrates with respect to the prior, it might be more efficient to choose design points that are somehow adapted to the posterior measure instead of space-filling designs that have a small fill distance. For example, we could use the sequential designs in \cite{sn16}. It would be interesting to prove suitable error bounds in this case, maybe using ideas from \cite{ws93}.}

{ In practical applications of Gaussian process emulators, such as in \cite{hkccr04}, the derivation of the emulator is often more involved than the simple approach presented in section \ref{sec:gp}. The hyper-parameters in the covariance kernel of the emulator are often unknown, and there is often a discrepancy between the mathematical model of the forward map and the true physical process, known as model error. These are both important issues for which the assumptions in our error bounds have not yet been verified. }

\bibliographystyle{siam}
\bibliography{bibgp}

\begin{thebibliography}{10}

\bibitem{adler}
{\sc R.~Adler}, {\em The Geometry of Random Fields}, John Wiley, 1981.

\bibitem{ar09}
{\sc C.~Andrieu and G.~O. Roberts}, {\em {The pseudo-marginal approach for
  efficient Monte Carlo computations}}, The Annals of Statistics,  (2009),
  pp.~697--725.

\bibitem{aronszajn50}
{\sc N.~Aronszajn}, {\em {Theory of Reproducing Kernels}}, Transactions of the
  American Mathematical Society, 68 (1950), pp.~337--404.

\bibitem{akksstv06}
{\sc S.~Arridge, J.~Kaipio, V.~Kolehmainen, M.~Schweiger, E.~Somersalo,
  T.~Tarvainen, and M.~Vauhkonen}, {\em Approximation errors and model
  reduction with an application in optical diffusion tomography}, Inverse
  Problems, 22 (2006), p.~175.

\bibitem{bnt10}
{\sc I.~Babu{\v{s}}ka, F.~Nobile, and R.~Tempone}, {\em A stochastic
  collocation method for elliptic partial differential equations with random
  input data}, SIAM review, 52 (2010), pp.~317--355.

\bibitem{bzkl13}
{\sc I.~Bilionis, N.~Zabaras, B.~A. Konomi, and G.~Lin}, {\em Multi-output
  separable {G}aussian process: Towards an efficient, fully {B}ayesian paradigm
  for uncertainty quantification}, Journal of Computational Physics, 241
  (2013), pp.~212--239.

\bibitem{brsrwm08}
{\sc N.~Bliznyuk, D.~Ruppert, C.~Shoemaker, R.~Regis, S.~Wild, and
  P.~Mugunthan}, {\em Bayesian calibration and uncertainty analysis for
  computationally expensive models using optimization and radial basis function
  approximation}, Journal of Computational and Graphical Statistics, 17 (2008).

\bibitem{bogachev}
{\sc V.~I. Bogachev}, {\em Gaussian Measures}, vol.~62 of Mathematical Surveys
  and Monographs, American Mathematical Society, 1998.

\bibitem{bwg08}
{\sc T.~Bui-Thanh, K.~Willcox, and O.~Ghattas}, {\em Model reduction for
  large-scale systems with high-dimensional parametric input space}, SIAM
  Journal on Scientific Computing, 30 (2008), pp.~3270--3288.

\bibitem{cohen2011analytic}
{\sc A.~Cohen, R.~Devore, and C.~Schwab}, {\em Analytic regularity and
  polynomial approximation of parametric and stochastic elliptic pde's},
  Analysis and Applications, 9 (2011), pp.~11--47.

\bibitem{cmps14}
{\sc P.~R. Conrad, Y.~M. Marzouk, N.~S. Pillai, and A.~Smith}, {\em
  Asymptotically exact {MCMC} algorithms via local approximations of
  computationally intensive models}, arXiv preprint arXiv:1402.1694,  (2014).

\bibitem{constantine2015active}
{\sc P.~G. Constantine}, {\em Active subspaces: Emerging ideas for dimension
  reduction in parameter studies}, vol.~2, SIAM, 2015.

\bibitem{constantine2014active}
{\sc P.~G. Constantine, E.~Dow, and Q.~Wang}, {\em Active subspace methods in
  theory and practice: Applications to kriging surfaces}, SIAM Journal on
  Scientific Computing, 36 (2014), pp.~A1500--A1524.

\bibitem{crsw13}
{\sc S.~L. Cotter, G.~O. Roberts, A.~M. Stuart, and D.~White}, {\em {MCMC}
  methods for functions: modifying old algorithms to make them faster},
  Statistical Science, 28 (2013), pp.~424--446.

\bibitem{ds15}
{\sc M.~Dashti and A.M.Stuart}, {\em The {B}ayesian approach to inverse
  problems}, in Handbook of Uncertainty Quantification, R.~Ghanem, D.~Higdon,
  and H.~Owhadi, eds., Springer.

\bibitem{gc11}
{\sc M.~Girolami and B.~Calderhead}, {\em Riemann manifold {L}angevin and
  {H}amiltonian {M}onte {C}arlo methods}, Journal of the Royal Statistical
  Society: Series B (Statistical Methodology), 73 (2011), pp.~123--214.

\bibitem{hs13}
{\sc M.~Hansen and C.~Schwab}, {\em Sparse adaptive approximation of high
  dimensional parametric initial value problems}, Vietnam Journal of
  Mathematics, 41 (2013), pp.~181--215.

\bibitem{hastings70}
{\sc W.~Hastings}, {\em Monte-{C}arlo sampling methods using {M}arkov chains
  and their applications}, Biometrika, 57 (1970), pp.~97--109.

\bibitem{hkccr04}
{\sc D.~Higdon, M.~Kennedy, J.~C. Cavendish, J.~A. Cafeo, and R.~D. Ryne}, {\em
  Combining field data and computer simulations for calibration and
  prediction}, SIAM Journal on Scientific Computing, 26 (2004), pp.~448--466.

\bibitem{kaipio2005statistical}
{\sc J.~P. Kaipio and E.~Somersalo}, {\em Statistical and Computational Inverse
  Problems}, vol.~160, Springer, 2005.

\bibitem{kennedy2001bayesian}
{\sc M.~C. Kennedy and A.~O'Hagan}, {\em Bayesian calibration of computer
  models}, Journal of the Royal Statistical Society: Series B (Statistical
  Methodology), 63 (2001), pp.~425--464.

\bibitem{mx09}
{\sc Y.~Marzouk and D.~Xiu}, {\em A stochastic collocation approach to
  {B}ayesian inference in inverse problems}, Communications in Computational
  Physics, 6 (2009), pp.~826--847.

\bibitem{mnr07}
{\sc Y.~M. Marzouk, H.~N. Najm, and L.~A. Rahn}, {\em Stochastic spectral
  methods for efficient {B}ayesian solution of inverse problems}, Journal of
  Computational Physics, 224 (2007), pp.~560--586.

\bibitem{matern}
{\sc B.~Mat{\'e}rn}, {\em Spatial Variation}, vol.~36, Springer Science \&
  Business Media, 2013.

\bibitem{mercer09}
{\sc J.~Mercer}, {\em Functions of positive and negative type, and their
  connection with the theory of integral equations}, Philosophical transactions
  of the royal society of London, series A, 209 (1909), pp.~415--446.

\bibitem{mrrtt53}
{\sc N.~Metropolis, A.~Rosenbluth, M.~Rosenbluth, A.~Teller, and E.~Teller},
  {\em Equation of state calculations by fast computing machines}, The J. of
  Chemical Physics, 21 (1953), p.~1087.

\bibitem{nww05}
{\sc F.~Narcowich, J.~Ward, and H.~Wendland}, {\em Sobolev bounds on functions
  with scattered zeros, with applications to radial basis function surface
  fitting}, Mathematics of Computation, 74 (2005), pp.~743--763.

\bibitem{nww06}
{\sc F.~J. Narcowich, J.~D. Ward, and H.~Wendland}, {\em Sobolev error
  estimates and a {B}ernstein inequality for scattered data interpolation via
  radial basis functions}, Constructive Approximation, 24 (2006), pp.~175--186.

\bibitem{niederreiter}
{\sc H.~Niederreiter}, {\em {Random Number Generation and quasi-Monte Carlo
  methods}}, SIAM, 1994.

\bibitem{o2006bayesian}
{\sc A.~O'Hagan}, {\em Bayesian analysis of computer code outputs: a tutorial},
  Reliability Engineering \& System Safety, 91 (2006), pp.~1290--1300.

\bibitem{daprato_zabczyk}
{\sc G.~D. Prato and J.~Zabczyk.}, {\em Stochastic Equations in Infinite
  Dimensions}, vol.~44 of Encyclopedia Math. Appl., Cambridge University Press,
  Cambridge, 1992.

\bibitem{rasmussen_williams}
{\sc C.~E. Rasmussen and C.~K. Williams}, {\em Gaussian processes for machine
  learning}, MIT Press, 2006.

\bibitem{rh15}
{\sc P.~Rebeschini and R.~Van~Handel}, {\em Can local particle filters beat the
  curse of dimensionality?}, The Annals of Applied Probability, 25 (2015),
  pp.~2809--2866.

\bibitem{robert_casella}
{\sc C.~Robert and G.~Casella}, {\em Monte {C}arlo {S}tatistical {M}ethods},
  Springer, 1999.

\bibitem{rudin}
{\sc W.~Rudin}, {\em Principles of Mathematical Analysis}, McGraw-Hill New,
  1964.

\bibitem{sacks1989design}
{\sc J.~Sacks, W.~J. Welch, T.~J. Mitchell, and H.~P. Wynn}, {\em Design and
  analysis of computer experiments}, Statistical science,  (1989),
  pp.~409--423.

\bibitem{sss13}
{\sc M.~Scheuerer, R.~Schaback, and M.~Schlather}, {\em {Interpolation of
  spatial data--A stochastic or a deterministic problem?}}, European Journal of
  Applied Mathematics, 24 (2013), pp.~601--629.

\bibitem{ss14}
{\sc C.~Schillings and C.~Schwab}, {\em Sparsity in {Bayesian} inversion of
  parametric operator equations}, Inverse Problems, 30 (2014), p.~065007.

\bibitem{sn16}
{\sc M.~Sinsbeck and W.~Nowak}, {\em Sequential design of computer experiments
  for the solution of {B}ayesian inverse problems with process emulators},
  tech. rep., University of Stuttgart, 2016.
\newblock Submitted to SIAM Journal on Uncertainty Quantification.

\bibitem{stein}
{\sc M.~L. Stein}, {\em {Interpolation of Spatial Data: Some Theory for
  Kriging}}, Springer, 1999.

\bibitem{stuart10}
{\sc A.~M. Stuart}, {\em Inverse problems}, vol.~19 of Acta Numerica, Cambridge
  University Press, 2010, pp.~451--559.

\bibitem{walter}
{\sc W.~Walter}, {\em Ordinary Differential Equations}, vol.~182 of Graduate
  Texts in Mathematics, Springer, 1998.

\bibitem{wendland}
{\sc H.~Wendland}, {\em {Scattered Data Approximation}}, vol.~17, Cambridge
  University Press, 2004.

\bibitem{ws93}
{\sc Z.-m. Wu and R.~Schaback}, {\em Local error estimates for radial basis
  function interpolation of scattered data}, IMA journal of Numerical Analysis,
  13 (1993), pp.~13--27.

\bibitem{xk03}
{\sc D.~Xiu and G.~E. Karniadakis}, {\em Modeling uncertainty in flow
  simulations via generalized polynomial chaos}, Journal of computational
  physics, 187 (2003), pp.~137--167.

\end{thebibliography}

\end{document}